\documentclass[10pt, a4paper]{amsart}
\usepackage{amscd,amsmath,amssymb,amsfonts,amsthm,ascmac}
\usepackage{ytableau,enumerate,mathrsfs,stmaryrd,latexsym} 
\usepackage[all]{xy}
\usepackage[top=27truemm,bottom=21truemm,left=26truemm,right=26truemm]{geometry}
\usepackage{color,graphicx}
\def\r{\mathbb{R}}
\def\c{\mathbb{C}}

\def\q{\mathbb{Q}}

\def\z{\mathbb{Z}}

\def\a{\alpha}

\newtheorem{thm}{Theorem}[section]
\newtheorem{defi}[thm]{Definition}

\newtheorem{rem}[thm]{Remark}
\newtheorem{prop}[thm]{Proposition}
\newtheorem{ex}[thm]{Example}

\newtheorem{cor}[thm]{Corollary}
\newtheorem{lem}[thm]{Lemma}

\newtheorem*{theorem 1}{Theorem 1}
\newtheorem*{theorem}{Theorem}

\newtheorem*{corollary}{Corollary}
\title[NEWTON-OKOUNKOV CONVEX BODIES AND POLYHEDRAL REALIZATIONS OF CRYSTAL BASES]{\fontsize{11pt}{11pt}\selectfont NEWTON-OKOUNKOV CONVEX BODIES OF SCHUBERT VARIETIES AND POLYHEDRAL REALIZATIONS OF CRYSTAL BASES}
\date{}

\author[N. Fujita]{\fontsize{10pt}{10pt}\selectfont Naoki Fujita}
\address[N. Fujita]{Department of Mathematics, Tokyo Institute of Technology, 2-12-1 Oh-okayama, Meguro-ku, Tokyo 152-8551, Japan}
\email{fujita.n.ac@m.titech.ac.jp}

\author[S. Naito]{\fontsize{10pt}{10pt}\selectfont Satoshi Naito}
\address[S. Naito]{Department of Mathematics, Tokyo Institute of Technology, 2-12-1 Oh-okayama, Meguro-ku, Tokyo 152-8551, Japan}
\email{naito@math.titech.ac.jp}

\begin{document} 
\subjclass[2010]{Primary 17B37; Secondary 05E10, 14M15, 14M25}
\keywords{Newton-Okounkov bodies, Schubert varieties, Nakashima-Zelevinsky polyhedral realizations, Kashiwara's involution}
\begin{abstract}
A Newton-Okounkov convex body is a convex body constructed from a projective variety with a valuation on its homogeneous coordinate ring; this is deeply connected with representation theory. For instance, the Littelmann string polytopes and the Feigin-Fourier-Littelmann-Vinberg polytopes are examples of Newton-Okounkov convex bodies. In this paper, we prove that the Newton-Okounkov convex body of a Schubert variety with respect to a specific valuation is identical to the Nakashima-Zelevinsky polyhedral realization of a Demazure crystal. As an application of this result, we show that Kashiwara's involution ($\ast$-operation) corresponds to a change of valuations on the rational function field.
\end{abstract}

\maketitle
\ytableausetup{smalltableaux}
\setcounter{tocdepth}{2}
\tableofcontents
\section{Introduction}
\indent
The study of Newton-Okounkov convex bodies is started by Okounkov in order to study multiplicity functions for representations of a reductive group (\cite{O1}, \cite{O2}), and afterward developed independently by Kaveh-Khovanskii (\cite{KK1}) and by Lazarsfeld-Mustata (\cite{LM}). These convex bodies generalize Newton polytopes for toric varieties to arbitrary projective varieties, and have various kinds of information about corresponding projective varieties; for instance, we can systematically construct a series of toric degenerations (see \cite[Corollary 3.14]{HK} and \cite[Theorem 1]{A}). Hence it is important to describe the explicit form of a Newton-Okounkov convex body. As an application of our main result below, we obtain a series of examples, in which a Newton-Okounkov convex body is given by a system of explicit affine inequalities.

A remarkable fact is that the theory of Newton-Okounkov convex bodies of Schubert varieties and Bott-Samelson varieties is deeply connected with representation theory. For instance, Kaveh (\cite{Kav}) proved that the Newton-Okounkov convex body of a Schubert variety with respect to a specific valuation is identical to the Littelmann string polytope constructed from the string parameterization for a Demazure crystal. Furthermore, one of the authors (\cite{F}) extended this result of Kaveh to Bott-Samelson varieties. Also, the Feigin-Fourier-Littelmann-Vinberg polytopes are examples of Newton-Okounkov convex bodies (see \cite{FFL} and \cite{Kir}). The purpose of this paper is to realize the Newton-Okounkov convex body of a Schubert variety with respect to a specific valuation as the Nakashima-Zelevinsky polyhedral realization of a Demazure crystal.

To be more precise, let $G$ be a connected, simply-connected semisimple algebraic group over $\c$, $\mathfrak{g}$ its Lie algebra, $W$ the Weyl group, and $E_i, F_i, h_i \in \mathfrak{g}$, $i \in I$, the Chevalley generators, where $I$ denotes an index set for the vertices of the Dynkin diagram. Choose a Borel subgroup $B \subset G$, and denote by $X(w) \subset G/B$ the Schubert variety corresponding to $w \in W$. A dominant integral weight $\lambda$ gives a line bundle $\mathcal{L}_\lambda$ on $G/B$; by restricting this bundle, we obtain a line bundle on $X(w)$, which we denote by the same symbol $\mathcal{L}_\lambda$. From the Borel-Weil theorem, we know that the space $H^0(X(w), \mathcal{L}_\lambda)$ of global sections is a $B$-module isomorphic to the dual module $V_w(\lambda)^\ast$; here, $V_w(\lambda)$ denotes the Demazure module corresponding to $w$. We take a valuation $v$ on the function field $\c(X(w))$ with values in $\z^r$, where $r := \ell(w)$, the length of $w$, and fix a nonzero section $\tau \in H^0(X(w), \mathcal{L}_\lambda)$; in addition, we assume that $v$ has one-dimensional leaves (see \S\S 3.1 for the definition). From these data, we construct a semigroup $S(X(w), \mathcal{L}_\lambda, v, \tau) \subset \z_{>0} \times \z^r$ (see Definition \ref{Newton-Okounkov convex body}). If we denote by $C(X(w), \mathcal{L}_\lambda, v, \tau) \subset \r_{\ge 0} \times \r^r$ the smallest real closed cone containing $S(X(w), \mathcal{L}_\lambda, v, \tau)$, then the Newton-Okounkov convex body $\Delta(X(w), \mathcal{L}_\lambda, v, \tau) \subset \r^r$ is defined to be the intersection of $C(X(w), \mathcal{L}_\lambda, v, \tau)$ and $\{1\} \times \r^r$. Now we restrict ourselves to a specific valuation. Let us take a reduced word ${\bf i} = (i_r, \ldots, i_1)$ for $w$. By using the birational morphism \[\c^r \rightarrow X(w),\ (t_r, \ldots, t_1) \mapsto \exp(t_r F_{i_r}) \exp(t_{r-1} F_{i_{r-1}}) \cdots \exp(t_1 F_{i_1}) \bmod B,\] we identify the function field $\c(X(w))$ with the rational function field $\c(t_r, \ldots, t_1)$. Define a valuation $v_{\bf i}$ on $\c(X(w))$ to be the highest term valuation on $\c(t_r, \ldots, t_1)$ with respect to the lexicographic order $t_1 > \cdots > t_r$ (see Example \ref{highest term valuation}). For a dominant integral weight $\lambda$, let $\tau_\lambda \in H^0(G/B, \mathcal{L}_\lambda)$ denote the lowest weight vector; by restricting this section, we obtain a section in $H^0(X(w), \mathcal{L}_\lambda)$, which we denote by the same symbol $\tau_\lambda$. In this setting, we study the Newton-Okounkov convex body $\Delta(X(w), \mathcal{L}_\lambda, v_{\bf i}, \tau_\lambda)$.

Let $U_q(\mathfrak{g})$ be the quantized enveloping algebra, and $\mathcal{B}(\infty)$ the crystal basis of the negative part $U_q(\mathfrak{u}^-)$. Denote by $\mathcal{B}(\lambda)$ the crystal basis of the irreducible highest weight $U_q(\mathfrak{g})$-module $V_q(\lambda)$ with highest weight $\lambda$, and by $\mathcal{B}_w(\lambda) \subset \mathcal{B}(\lambda)$ the Demazure crystal corresponding to $w \in W$. In the theory of crystal bases, it is important to give their concrete realizations. Until now, many useful realizations have been discovered; the theory of Nakashima-Zelevinsky polyhedral realizations is one of them. Take an infinite sequence of indices $\tilde{\bf i} = (\ldots, i_k, \ldots, i_2, i_1)$ such that $i_k \neq i_{k+1}$ for all $k \ge 1$, and such that the cardinality of $\{k \ge 1 \mid i_k = i\}$ is $\infty$ for each $i \in I$. Then, we can associate to $\tilde{\bf i}$ a crystal structure on $\z^{\infty} := \{(\ldots, a_k, \ldots, a_2, a_1) \mid a_k \in \z\ {\rm and}\ a_k = 0\ {\rm for}\ k \gg 0\}$, and obtain an embedding of crystals $\Psi_{\tilde{\bf i}}: \mathcal{B}(\infty) \hookrightarrow \z^\infty$, called the Kashiwara embedding with respect to $\tilde{\bf i}$ (see \S\S 2.3). Nakashima-Zelevinsky (\cite{NZ}) described explicitly the image of $\mathcal{B}(\infty)$ under a certain positivity assumption on $\tilde{\bf i}$. Afterward, Nakashima (\cite{N1}, \cite{N2}) gave a similar description of the Demazure crystal $\mathcal{B}_w(\lambda)$ under the assumption that $(\tilde{\bf i}, \lambda)$ is ample (see Definition \ref{definition of ample}). These descriptions of crystal bases are called polyhedral realizations. Let ${\bf i} = (i_r, \ldots, i_1)$ be a reduced word for $w \in W$, and extend it to a reduced word $\hat{\bf i} = (i_N, \ldots, i_{r+1}, i_r, \ldots, i_1)$ for the longest element $w_0 \in W$. We consider the Kashiwara embedding $\Psi_{\tilde{\bf i}}$ with respect to an extension $\tilde{\bf i} = (\ldots, i_k, \ldots, i_{N+1}, i_N, \ldots, i_2, i_1)$ of $\hat{\bf i}$. In this case, the image of $\mathcal{B}(\infty)$ (resp., $\mathcal{B}_w(\lambda)$) can be regarded as a subset of $\z^N$ (resp., $\z^r$) (see Remark \ref{definition1} and Definition \ref{definition2}); we denote by $\Psi_{\bf i}: \mathcal{B}(\infty) \hookrightarrow \z^N$ and $\Psi_{\bf i} ^{(\lambda, w)}: \mathcal{B}_w(\lambda) \hookrightarrow \z^r$ the corresponding embeddings. From the embedding $\Psi_{\bf i} ^{(\lambda, w)}$, we obtain a subset $\mathcal{S}_{\bf i} ^{(\lambda, w)} \subset \z_{>0} \times \z^r$ (see Definition \ref{definition3}). Let us denote by $\mathcal{C}_{\bf i} ^{(\lambda, w)} \subset \r_{\ge 0} \times \r^r$ the smallest real closed cone containing $\mathcal{S}_{\bf i} ^{(\lambda, w)}$, and by $\Delta_{\bf i} ^{(\lambda, w)} \subset \r^r$ the intersection of $\mathcal{C}_{\bf i} ^{(\lambda, w)}$ and $\{1\} \times \r^r$. By using the theory of generalized string polytopes (see \cite{F} for the definition), we deduce that $\Delta_{\bf i} ^{(\lambda, w)} \cap \z^r = \Psi_{\bf i} ^{(\lambda, w)}(\mathcal{B}_w(\lambda))$ (see Corollary \ref{a finite union of polytopes}); here, we do not necessarily assume that $(\tilde{\bf i}, \lambda)$ is ample.

In order to relate the Newton-Okounkov convex body $\Delta(X(w), \mathcal{L}_\lambda, v_{\bf i}, \tau_\lambda)$ with the polyhedral realization $\Delta_{\bf i} ^{(\lambda, w)}$, we make use of global crystal bases. Let us denote by $\{G^{\rm low} _\lambda(b) \mid b \in \mathcal{B}_w(\lambda)\} \subset V_w(\lambda)$ the specialization of the lower global basis at $q = 1$, and by $\{G^{\rm up} _{\lambda, w}(b) \mid b \in \mathcal{B}_w(\lambda)\} \subset H^0(X(w), \mathcal{L}_\lambda) = V_w(\lambda)^\ast$ the dual basis (see \S\S 4.1). The following is the main result of this paper.

\vspace{2mm}\begin{theorem}
Let ${\bf i}$ be a reduced word for $w \in W$, and $\lambda$ a dominant integral weight.
\begin{enumerate}
\item[{\rm (1)}] The Kashiwara embedding $\Psi_{\bf i} ^{(\lambda, w)}(b)$ is equal to $-v_{\bf i}(G^{\rm up} _{\lambda, w}(b)/\tau_\lambda)$ for all $b \in \mathcal{B}_w(\lambda)$.
\item[{\rm (2)}] The polyhedral realization $\Delta_{\bf i} ^{(\lambda, w)}$ is identical to $-\Delta(X(w), \mathcal{L}_\lambda, v_{\bf i}, \tau_\lambda)$.
\end{enumerate}
\end{theorem}\vspace{2mm}

\S\S 5.1 is devoted to the study of explicit forms of Newton-Okounkov convex bodies. To be more precise, under the assumption that $(\tilde{\bf i}, \lambda)$ is ample, Nakashima's description of $\Psi_{\bf i} ^{(\lambda, w)}(\mathcal{B}_w(\lambda))$ also gives a system of explicit affine inequalities defining the Newton-Okounkov convex body $\Delta_{\bf i} ^{(\lambda, w)} = -\Delta(X(w), \mathcal{L}_\lambda, v_{\bf i}, \tau_\lambda)$. 

In \S\S 5.2, we compare our result with the result of \cite{Kav}. To be more precise, let us define a valuation $\tilde{v}_{\bf i}$ on $\c(X(w))$ to be the highest term valuation on $\c(t_r, \ldots, t_1)$ with respect to the lexicographic order $t_r \succ \cdots \succ t_1$ (see Example \ref{highest term valuation}). Kaveh (\cite{Kav}) proved that the value $-\tilde{v}_{\bf i}(G^{\rm up} _{\lambda, w}(b)/\tau_\lambda)$ for $b \in \mathcal{B}_w(\lambda)$ is equal to the string parameterization of $b$ with respect to ${\bf i}$, and that the Newton-Okounkov convex body $-\Delta(X(w), \mathcal{L}_\lambda, \tilde{v}_{\bf i}, \tau_\lambda)$ is identical to the Littelmann string polytope $\widetilde{\Delta}_{\bf i} ^{(\lambda, w)}$. Let us consider the case that $w = w_0$. In this case, the Schubert variety $X(w_0)$ is just the full flag variety $G/B$, and the Demazure crystal $\mathcal{B}_{w_0}(\lambda)$ is just the crystal basis $\mathcal{B}(\lambda)$. We denote the section $G^{\rm up} _{\lambda, w_0}(b) \in H^0(G/B, \mathcal{L}_\lambda)$ for $b \in \mathcal{B}(\lambda)$ simply by $G^{\rm up} _{\lambda}(b)$. For a reduced word ${\bf i} = (i_N, \ldots, i_1)$ for $w_0$, our result combined with the result of Kaveh above implies that Kashiwara's involution $\ast: \mathcal{B}(\infty) \rightarrow \mathcal{B}(\infty)$ corresponds to the change of valuations from $v_{\bf i}$ to $\tilde{v}_{{\bf i}^{\rm op}}$, which gives a geometric interpretation of $\ast$; here we write ${\bf i}^{\rm op} := (i_1, \ldots, i_N)$. More precisely, we obtain the following.

\vspace{2mm}\begin{corollary}
Let ${\bf i}$ be a reduced word for $w_0$, and $N := \ell(w_0)$. Then, there exists a piecewise-linear map $\eta_{\bf i}: \r^N \rightarrow \r^N$ satisfying the following conditions$:$
\begin{enumerate}
\item[{\rm (i)}] The map $\eta_{\bf i}$ corresponds to Kashiwara's involution $\ast$ through the Kashiwara embedding $\Psi_{\bf i}$$:$ \[b^\ast = \Psi_{\bf i} ^{-1} \circ \eta_{\bf i} \circ \Psi_{\bf i}(b)\] for all $b \in \mathcal{B}(\infty)$.
\item[{\rm (ii)}] The map $\eta_{\bf i}$ corresponds to the change of valuations from $v_{\bf i}$ to $\tilde{v}_{{\bf i}^{\rm op}}$$:$ \[\eta_{\bf i}(-v_{\bf i}(G^{\rm up} _{\lambda}(b)/\tau_{\lambda})) = -\tilde{v}_{{\bf i}^{\rm op}}(G^{\rm up} _{\lambda}(b)/\tau_{\lambda})\] for every dominant integral weight $\lambda$ and $b \in \mathcal{B}(\lambda)$.
\item[{\rm (iii)}] The map $\eta_{\bf i}$ induces a bijective piecewise-linear map from the polyhedral realization $\Delta_{\bf i} ^{(\lambda, w_0)} = -\Delta(G/B, \mathcal{L}_\lambda, v_{\bf i}, \tau_{\lambda})$ onto the string polytope $\widetilde{\Delta}_{{\bf i}^{\rm op}} ^{(\lambda, w_0)} = -\Delta(G/B, \mathcal{L}_\lambda, \tilde{v}_{{\bf i}^{\rm op}}, \tau_{\lambda})$ for every dominant integral weight $\lambda$.
\end{enumerate}
\end{corollary}

\section{Polyhedral realizations of crystal bases}
\subsection{Abstract crystals}

First of all, we recall the definition of abstract crystals introduced in \cite{Kas4}. Let $G$ be a connected, simply-connected semisimple algebraic group over $\c$ of rank $n$, $\mathfrak{g}$ its Lie algebra, and $W$ the Weyl group. Choose a Borel subgroup $B \subset G$ and a maximal torus $T \subset B$. Denote by $\mathfrak{t}$ the Lie algebra of $T$, by $\mathfrak{t}^\ast := {\rm Hom}_\c (\mathfrak{t}, \c)$ its dual space, and by $\langle \cdot, \cdot \rangle: \mathfrak{t}^\ast \times \mathfrak{t} \rightarrow \c$ the canonical pairing. Let $\{\alpha_i \mid i \in I\} \subset \mathfrak{t}^\ast$ be the set of simple roots, $\{h_i \mid i \in I\} \subset \mathfrak{t}$ the set of simple coroots, and $P \subset \mathfrak{t}^\ast$ the weight lattice, where $I := \{1, 2, \ldots, n\}$, an index set for the vertices of the Dynkin diagram. 

\vspace{2mm}\begin{defi}\normalfont
A {\it crystal} $\mathcal{B}$ is a set equipped with the maps 
\begin{enumerate}
\item[] ${\rm wt}: \mathcal{B} \rightarrow P$,
\item[] $\varepsilon_i: \mathcal{B} \rightarrow \z \cup \{-\infty\}$, $\varphi_i: \mathcal{B} \rightarrow \z \cup \{-\infty\}\ {\rm for}\ i \in I$,\ and
\item[] $\tilde{e}_i: \mathcal{B} \rightarrow \mathcal{B} \cup \{0\}$, $\tilde{f}_i: \mathcal{B} \rightarrow \mathcal{B} \cup \{0\}\ {\rm for}\ i \in I$,
\end{enumerate}
satisfying the following conditions:
\begin{enumerate}
\item[(i)] $\varphi_i(b) = \varepsilon_i(b) + \langle{\rm wt}(b), h_i\rangle$ for $i \in I$,
\item[(ii)] ${\rm wt}(\tilde{e}_i b) = {\rm wt}(b) + \alpha_i$, $\varepsilon_i(\tilde{e}_i b) = \varepsilon_i(b) -1$, and $\varphi_i(\tilde{e}_i b) = \varphi_i(b) +1$ for $i \in I$ and $b \in \mathcal{B}$ such that $\tilde{e}_i b \in \mathcal{B}$,
\item[(iii)] ${\rm wt}(\tilde{f}_i b) = {\rm wt}(b) - \alpha_i$, $\varepsilon_i(\tilde{f}_i b) = \varepsilon_i(b) +1$, and $\varphi_i(\tilde{f}_i b) = \varphi_i(b) -1$ for $i \in I$ and $b \in \mathcal{B}$ such that $\tilde{f}_i b \in \mathcal{B}$,
\item[(iv)] $b^\prime = \tilde{e}_i b$ if and only if $b = \tilde{f}_i b^\prime$ for $i \in I$ and $b, b^\prime \in \mathcal{B}$,
\item[(v)] $\tilde{e}_i b = \tilde{f}_i b = 0$ for $i \in I$ and $b \in \mathcal{B}$ such that $\varphi_i(b) = -\infty$.
\end{enumerate}
Here, $- \infty$ and $0$ are additional elements that are not contained in $\z$ and $\mathcal{B}$, respectively.
\end{defi}\vspace{2mm}

\begin{defi}\normalfont
Let $\mathcal{B}_1, \mathcal{B}_2$ be two crystals. A map \[\psi: \mathcal{B}_1 \cup \{0\} \rightarrow \mathcal{B}_2 \cup \{0\}\] is called a {\it strict morphism} of crystals from $\mathcal{B}_1$ to $\mathcal{B}_2$ if it satisfies the following conditions:
\begin{enumerate}
\item[(i)] $\psi(0) = 0$,
\item[(ii)] ${\rm wt}(\psi(b)) = {\rm wt}(b)$, $\varepsilon_i(\psi(b)) = \varepsilon_i(b)$, and $\varphi_i(\psi(b)) = \varphi_i(b)$ for $i \in I$ and $b \in \mathcal{B}_1$ such that $\psi(b) \in \mathcal{B}_2$,
\item[(iii)] $\tilde{e}_i \psi(b) = \psi(\tilde{e}_i b)$ and $\tilde{f}_i \psi(b) = \psi(\tilde{f}_i b)$ for $i \in I$ and $b \in \mathcal{B}_1$;
\end{enumerate}
here, if $\psi(b) = 0$, then we set $\tilde{e}_i \psi(b) = \tilde{f}_i \psi(b) = 0$. An injective strict morphism is called a {\it strict embedding} of crystals. 
\end{defi}\vspace{2mm}

Consider the total order $<$ on $\z \cup \{-\infty\}$ given by the usual order on $\z$, and by $-\infty < s$ for all $s \in \z$. For two crystals $\mathcal{B}_1, \mathcal{B}_2$, we can define another crystal $\mathcal{B}_1 \otimes \mathcal{B}_2$ (called the tensor product) as follows: 
\begin{align*}
&\mathcal{B}_1 \otimes \mathcal{B}_2 := \{b_1 \otimes b_2 \mid b_1 \in \mathcal{B}_1,\ b_2 \in \mathcal{B}_2\},\\
&{\rm wt}(b_1 \otimes b_2) := {\rm wt}(b_1) + {\rm wt}(b_2),\\
&\varepsilon_i(b_1 \otimes b_2) := \max\{\varepsilon_i(b_1),\ \varepsilon_i(b_2) - \langle{\rm wt}(b_1), h_i\rangle\},\\
&\varphi_i(b_1 \otimes b_2) := \max\{\varphi_i(b_2),\ \varphi_i(b_1) + \langle{\rm wt}(b_2), h_i\rangle\},\\
&\tilde{e}_i(b_1 \otimes b_2) :=
\begin{cases}
\tilde{e}_i b_1 \otimes b_2 &{\rm if}\ \varphi_i (b_1) \ge \varepsilon_i (b_2),\\
b_1 \otimes \tilde{e}_i b_2 &{\rm if}\ \varphi_i (b_1) < \varepsilon_i (b_2),
\end{cases}\\
&\tilde{f}_i(b_1 \otimes b_2) :=
\begin{cases}
\tilde{f}_i b_1 \otimes b_2 &{\rm if}\ \varphi_i (b_1) > \varepsilon_i (b_2),\\
b_1 \otimes \tilde{f}_i b_2 &{\rm if}\ \varphi_i (b_1) \le \varepsilon_i (b_2).
\end{cases}
\end{align*}
Here, $b_1 \otimes b_2$ stands for an ordered pair $(b_1, b_2)$, and we set $b_1 \otimes 0 = 0 \otimes b_2 = 0$.

\vspace{2mm}\begin{ex}\normalfont\label{example of crystals}
For $\lambda \in P$, let $R_\lambda = \{r_\lambda\}$ be the crystal consisting of only one element given by: ${\rm wt}(r_\lambda) = \lambda$, $\varepsilon_i(r_\lambda) = - \langle\lambda, h_i\rangle$, $\varphi_i(r_\lambda) = 0$, and $\tilde{e}_i r_\lambda = \tilde{f}_i r_\lambda = 0$.
\end{ex}

\subsection{Crystal bases and global bases}

Here we recall some basic facts about crystal bases and global bases (following \cite{Kas2}, \cite{Kas3}, and \cite{Kas4}). Define a symmetric bilinear form $(\cdot, \cdot)$ on $\mathfrak{t}^\ast$ by $2(\alpha_j, \alpha_i)/(\alpha_i, \alpha_i) = \langle \alpha_j, h_i \rangle$ for all $i, j \in I$, and by $(\alpha_i, \alpha_i) = 2$ for all short simple roots $\alpha_i$. We set $(c_{i, j})_{i, j \in I} := (\langle \alpha_j, h_i \rangle)_{i, j \in I}$, the Cartan matrix of $\mathfrak{g}$, and also set
\begin{align*}
&q_i := q^{(\a_i, \a_i)/2}\ {\rm for}\ i \in I,\\
&[s]_i := \frac{q_i ^s - q_i ^{-s}}{q_i - q_i ^{-1}}\ {\rm for}\ i \in I,\ s \in \z,\\
&[s]_i ! := [s]_i [s-1]_i \cdots [1]_i\ {\rm for}\ i \in I,\ s \in \z_{\ge 0},\\
&\genfrac{[}{]}{0pt}{}{s}{k}_i := \frac{[s]_i [s-1]_i \cdots [s - k +1]_i}{[k]_i [k-1]_i \cdots [1]_i}\ {\rm for}\ i \in I,\ s, k \in \z_{\ge 0}\ {\rm such\ that}\ k \le s,
\end{align*}
where $[0]_i ! := 1$ and $\genfrac{[}{]}{0pt}{}{s}{0}_i := 1$. 
\vspace{2mm}\begin{defi}\normalfont
For a finite-dimensional semisimple Lie algebra $\mathfrak{g}$, the {\it quantized enveloping algebra} $U_q(\mathfrak{g})$ is the unital associative $\q(q)$-algebra with generators $e_i, f_i, t_i, t_i ^{-1}$, $i \in I$, and relations:
\begin{enumerate}
\item[(i)] $t_i t_i ^{-1}=1$ and $t_i t_j = t_j t_i$ for $i, j \in I$,
\item[(ii)] $t_i e_j t_i ^{-1} = q_i ^{c_{i, j}} e_j$ and $t_i f_j t_i ^{-1} = q_i ^{-c_{i, j}} f_j$ for $i, j \in I$,
\item[(iii)] $e_i f_j - f_j e_i = \delta_{i, j} (t_i - t_i ^{-1})/(q_i - q_i ^{-1})$ for $i, j \in I$,
\item[(iv)] $\sum_{s = 0} ^{1 - c_{i, j}}(-1)^s e_i^{(s)} e_j e_i ^{(1 - c_{i, j} - s)} = \sum_{s = 0} ^{1 - c_{i, j}}(-1)^s f_i^{(s)} f_j f_i ^{(1 - c_{i, j} - s)} = 0$ for $i, j \in I$ such that $i \neq j$.
\end{enumerate}
Here, $\delta_{i, j}$ denotes the Kronecker delta, and $e_i ^{(s)} := e_i ^s/[s]_i !$, $f_i ^{(s)} := f_i ^s/[s]_i !$ for $i \in I$, $s \in \z_{\ge 0}$. 
\end{defi}\vspace{2mm}

Let us denote by $U_q (\mathfrak{u})$ (resp., $U_q (\mathfrak{u}^-)$) the $\q(q)$-subalgebra of $U_q(\mathfrak{g})$ generated by $\{e_i\ |\ i \in I\}$ (resp., $\{f_i\ |\ i \in I\}$). Define a $\q$-algebra involution ${}^-$ and a $\q(q)$-algebra anti-involution $\ast$ on $U_q (\mathfrak{g})$ by
\begin{align*}
\overline{e}_i &= e_i,\ \overline{f}_i = f_i,\ \overline{t_i}=t_i ^{-1},\ \overline{q}=q^{-1},\ {\rm and}\\
e_i ^\ast &= e_i,\ f_i ^\ast = f_i,\ t_i ^\ast = t_i ^{-1};
\end{align*}
the involution ${}^-$ (resp., $\ast$) is called the {\it bar involution} (resp., {\it Kashiwara's involution}). Note that these preserve $U_q (\mathfrak{u})$ and $U_q (\mathfrak{u}^-)$; also, we have $\ast \circ {}^- = {}^- \circ \ast$. Denote by $A \subset \q(q)$ the $\q$-subalgebra of $\q(q)$ consisting of rational functions regular at $q = 0$. Let us take a free $A$-submodule $L(\infty) \subset U_q (\mathfrak{u}^-)$, a $\q$-basis $\mathcal{B}(\infty) \subset L(\infty)/qL(\infty)$, and define operators $\tilde{e}_i, \tilde{f}_i$, $i \in I$, on $U_q (\mathfrak{u}^-)$ as in \cite[\S\S 3.5]{Kas2} (note that we use the notation $\mathcal{B}(\infty)$ instead of $B(\infty)$). The operators $\tilde{e}_i, \tilde{f}_i$, $i \in I$, are called the {\it Kashiwara operators}.

\vspace{2mm}\begin{prop}[{see \cite[Theorem 4]{Kas2}}]
The following hold.
\begin{enumerate}
\item[{\rm (1)}] $\tilde{e}_i L(\infty) \subset L(\infty)$ and $\tilde{f}_i L(\infty) \subset L(\infty)$ for all $i \in I;$ hence $\tilde{e}_i, \tilde{f}_i$, $i \in I$, act on $L(\infty)/qL(\infty)$.
\item[{\rm (2)}] $\tilde{e}_i \mathcal{B}(\infty) \subset \mathcal{B}(\infty) \cup \{0\}$ and $\tilde{f}_i \mathcal{B}(\infty) \subset \mathcal{B}(\infty)$ for all $i \in I$. 
\item[{\rm (3)}] Define maps $\varepsilon_i, \varphi_i: \mathcal{B}(\infty) \rightarrow \z \cup \{-\infty\}$ for $i \in I$ by \[\varepsilon_i(b) := \max\{k \in \z_{\ge 0} \mid \tilde{e}_i ^k b \neq 0\},\ {\it and}\ \varphi_i (b) := \varepsilon_i (b) + \langle{\rm wt}(b), h_i\rangle.\] Then, the maps ${\rm wt}$, $\varepsilon_i$, $\varphi_i$, $\tilde{e}_i$, and $\tilde{f}_i$, $i \in I$, provide a crystal structure on $\mathcal{B}(\infty)$. 
\end{enumerate}
\end{prop}\vspace{2mm}

The pair $(L(\infty), \mathcal{B}(\infty))$ is called the {\it lower crystal basis} of $U_q (\mathfrak{u}^-)$. We see from \cite[Proposition 5.2.4]{Kas2} that $L(\infty)^\ast = L(\infty)$, and from \cite[Theorem 2.1.1]{Kas4} that $\mathcal{B}(\infty)^\ast = \mathcal{B}(\infty)$. Hence the maps \[{\rm wt},\ \varepsilon_i ^\ast := \varepsilon_i \circ \ast,\ \varphi_i ^\ast := \varphi_i \circ \ast,\ \tilde{e}_i ^\ast := \ast \circ \tilde{e}_i \circ \ast,\ {\rm and}\ \tilde{f}_i ^\ast := \ast \circ \tilde{f}_i \circ \ast,\ i \in I,\] provide another crystal structure on $\mathcal{B}(\infty)$. For a dominant integral weight $\lambda$, let $V_q (\lambda)$ denote the irreducible highest weight $U_q(\mathfrak{g})$-module with highest weight $\lambda$ over $\q(q)$, and $v_{q, \lambda} \in V_q (\lambda)$ the highest weight vector. We take a free $A$-submodule $L(\lambda) \subset V_q(\lambda)$, a $\q$-basis $\mathcal{B} (\lambda) \subset L(\lambda)/q L(\lambda)$, and operators $\tilde{e}_i, \tilde{f}_i$, $i \in I$, on $V_q (\lambda)$ as in \cite[Section 2]{Kas2} (note that we use the notation $\mathcal{B}(\lambda)$ instead of $B(\lambda)$); the operators $\tilde{e}_i, \tilde{f}_i$, $i \in I$, are also called the {\it Kashiwara operators}. 

\vspace{2mm}\begin{prop}[{see \cite[Theorem 2]{Kas2}}]
For a dominant integral weight $\lambda$, the following hold.
\begin{enumerate}
\item[{\rm (1)}] $\tilde{e}_i L(\lambda) \subset L(\lambda)$ and $\tilde{f}_i L(\lambda) \subset L(\lambda)$ for all $i \in I;$ hence $\tilde{e}_i, \tilde{f}_i$, $i \in I$, act on $L(\lambda)/qL(\lambda)$.
\item[{\rm (2)}] $\tilde{e}_i \mathcal{B}(\lambda) \subset \mathcal{B}(\lambda) \cup \{0\}$ and $\tilde{f}_i \mathcal{B}(\lambda) \subset \mathcal{B}(\lambda) \cup \{0\}$ for all $i \in I$. 
\item[{\rm (3)}] Define maps $\varepsilon_i, \varphi_i: \mathcal{B}(\lambda) \rightarrow \z \cup \{-\infty\}$ for $i \in I$ by \[\varepsilon_i(b) := \max\{k \in \z_{\ge 0} \mid \tilde{e}_i ^k b \neq 0\},\ {\it and}\ \varphi_i (b) := \max\{k \in \z_{\ge 0} \mid \tilde{f}_i ^k b \neq 0\}.\] Then, the maps ${\rm wt}$, $\varepsilon_i$, $\varphi_i$, $\tilde{e}_i$, and $\tilde{f}_i$, $i \in I$, provide a crystal structure on $\mathcal{B}(\lambda)$. 
\end{enumerate}
\end{prop}\vspace{2mm}

The pair $(L(\lambda), \mathcal{B}(\lambda))$ is called the {\it lower crystal basis} of $V_q (\lambda)$. The crystals $\mathcal{B}(\infty)$ and $\mathcal{B} (\lambda)$ are related as follows.

\vspace{2mm}\begin{prop}[{\cite[Theorem 5]{Kas2}}]\label{connection of crystals}
For a dominant integral weight $\lambda$, let $\pi_\lambda: U_q (\mathfrak{u}^-) \twoheadrightarrow V_q (\lambda)$ denote the surjective $U_q (\mathfrak{u}^-)$-module homomorphism given by $u \mapsto u \cdot v_{q, \lambda}$. 
\begin{enumerate}
\item[{\rm (1)}] The equality $\pi_\lambda (L(\infty)) = L (\lambda)$ holds$;$ hence $\pi_\lambda$ induces a surjective $\q$-linear map $L(\infty)/q L(\infty) \twoheadrightarrow L(\lambda)/q L(\lambda)$ $($denoted also by $\pi_\lambda)$.
\item[{\rm (2)}] The $\q$-linear map $\pi_\lambda$ induces a bijective map $\pi_\lambda: \widetilde{\mathcal{B}} (\lambda) \xrightarrow[]{\sim} \mathcal{B} (\lambda)$, where \[\widetilde{\mathcal{B}} (\lambda) := \{b \in \mathcal{B}(\infty) \mid \pi_\lambda(b) \neq 0\}.\]
\item[{\rm (3)}] $\tilde{f}_i \pi_\lambda(b) = \pi_\lambda (\tilde{f}_i b)$ for all $i \in I$, $b \in \mathcal{B}(\infty)$.
\item[{\rm (4)}] $\tilde{e}_i \pi_\lambda(b) = \pi_\lambda (\tilde{e}_i b)$ for all $i \in I$, $b \in \widetilde{\mathcal{B}}(\lambda)$.
\end{enumerate}
\end{prop}\vspace{2mm}

It follows easily from Proposition \ref{connection of crystals} (3), (4) that $\varepsilon_i (\pi_\lambda(b)) = \varepsilon_i (b)$, and that $\varphi_i (\pi_\lambda (b)) = \varphi_i (b) + \langle\lambda, h_i\rangle$ for all $b \in \widetilde{\mathcal{B}}(\lambda)$. 

Now we denote by $U_q ^\q (\mathfrak{u}^-)$ the $\q[q, q^{-1}]$-subalgebra of $U_q (\mathfrak{u}^-)$ generated by $\{f_i ^{(k)} \mid i \in I,\ k \in \z_{\ge 0}\}$, and set $V_q ^{\q} (\lambda) := \pi_\lambda (U_q ^\q (\mathfrak{u}^-))$; note that the $U_q ^\q (\mathfrak{u}^-)$-submodule $V_q ^{\q} (\lambda)$ is invariant under the action of $e_i$, $f_i$, and $(t_i - t_i ^{-1})/(q_i - q_i ^{-1})$ for all $i \in I$. Also, we define a $\q$-involution ${}^-$ on $V_q (\lambda)$ by $\overline{u \cdot v_{q, \lambda}}:= \overline{u} \cdot v_{q, \lambda}$ for $u \in U_q (\mathfrak{g})$. Then, the natural maps 
\begin{align*}
L(\infty) \cap \overline{L(\infty)} \cap U_q ^\q (\mathfrak{u}^-) &\rightarrow L(\infty)/q L(\infty)\ {\rm and}\\
L(\lambda) \cap \overline{L(\lambda)} \cap V_q ^\q (\lambda) &\rightarrow L(\lambda)/q L(\lambda)
\end{align*}
are isomorphisms of $\q$-vector spaces (\cite[Theorem 6]{Kas2}). If we denote by $G_q ^{\rm low}: L(\infty)/q L(\infty) \rightarrow L(\infty) \cap \overline{L(\infty)} \cap U_q ^\q (\mathfrak{u}^-)$ and $G^{\rm low} _{q, \lambda}: L(\lambda)/q L(\lambda) \rightarrow L(\lambda) \cap \overline{L(\lambda)} \cap V_q ^\q (\lambda)$ the inverses of these isomorphisms, respectively, then the set $\{G_q ^{\rm low} (b) \mid b \in \mathcal{B}(\infty)\}$ (resp., $\{G^{\rm low} _{q, \lambda} (b) \mid b \in \mathcal{B} (\lambda)\}$) forms a $\q[q, q^{-1}]$-basis of $U_q ^\q (\mathfrak{u}^-)$ (resp., $V_q ^\q (\lambda)$); this is called the {\it lower global basis} of $U_q (\mathfrak{u}^-)$ (resp., the {\it lower global basis} of $V_q (\lambda)$). The following is a fundamental property of these bases.

\vspace{2mm}\begin{prop}\label{properties of global bases}
For a dominant integral weight $\lambda$, the following hold.
\begin{enumerate}
\item[{\rm (1)}] $G_q ^{\rm low} (b)^\ast = G_q ^{\rm low} (b^\ast)$ for all $b \in \mathcal{B}(\infty)$.
\item[{\rm (2)}] $\pi_\lambda(G_q ^{\rm low} (b)) = G^{\rm low} _{q, \lambda} (\pi_\lambda (b))$ for all $b \in \mathcal{B}(\infty)$.
\item[{\rm (3)}] For all $i \in I$, $b \in \mathcal{B}(\lambda)$, and $k \in \z_{\ge 0}$,
\begin{align*}
&e_i ^{(k)} \cdot G^{\rm low} _{q, \lambda} (b) \in \genfrac{[}{]}{0pt}{}{\varphi_i (b) + k}{k}_i G^{\rm low} _{q, \lambda} (\tilde{e}_i ^k b) + \sum_{\substack{b^\prime \in \mathcal{B}(\lambda);\ {\rm wt}(b^\prime) = {\rm wt}(\tilde{e}_i ^k b),\\ \varphi_i (b^\prime) > \varphi_i (\tilde{e}_i ^k b)}} \z[q, q^{-1}] G^{\rm low} _{q, \lambda} (b^\prime),\ {\it and}\\
&f_i ^{(k)} \cdot G^{\rm low} _{q, \lambda} (b) \in \genfrac{[}{]}{0pt}{}{\varepsilon_i (b) + k}{k}_i G^{\rm low} _{q, \lambda} (\tilde{f}_i ^k b) + \sum_{\substack{b^\prime \in \mathcal{B}(\lambda);\ {\rm wt}(b^\prime) = {\rm wt}(\tilde{f}_i ^k b),\\ \varepsilon_i (b^\prime) > \varepsilon_i (\tilde{f}_i ^k b)}} \z[q, q^{-1}] G^{\rm low} _{q, \lambda} (b^\prime).
\end{align*}
\item[{\rm (4)}] For all $i \in I$, $b \in \mathcal{B}(\infty)$, and $k \in \z_{\ge 0}$,
\begin{align*}
&f_i ^{(k)} \cdot G_q ^{\rm low} (b) \in \genfrac{[}{]}{0pt}{}{\varepsilon_i (b) + k}{k}_i G_q ^{\rm low} (\tilde{f}_i ^k b) + \sum_{\substack{b^\prime \in \mathcal{B}(\infty);\ {\rm wt}(b^\prime) = {\rm wt}(\tilde{f}_i ^k b),\\ \varepsilon_i (b^\prime) > \varepsilon_i (\tilde{f}_i ^k b)}} \z[q, q^{-1}] G_q ^{\rm low} (b^\prime),\ {\it and}\\
&G_q ^{\rm low} (b) \cdot f_i ^{(k)} \in \genfrac{[}{]}{0pt}{}{\varepsilon_i ^\ast (b) + k}{k}_i G_q ^{\rm low} ((\tilde{f}_i ^\ast)^k b) + \sum_{\substack{b^\prime \in \mathcal{B}(\infty);\ {\rm wt}(b^\prime) = {\rm wt}((\tilde{f}_i ^\ast)^k b),\\ \varepsilon_i ^\ast (b^\prime) > \varepsilon_i ^\ast ((\tilde{f}_i ^\ast)^k b)}} \z[q, q^{-1}] G_q ^{\rm low} (b^\prime).
\end{align*}
\end{enumerate}
\end{prop}

\begin{proof}
It follows from the equality $\ast \circ {}^- = {}^- \circ \ast$ that $\overline{L(\infty)}$ and hence $L(\infty) \cap \overline{L(\infty)} \cap U_q ^\q (\mathfrak{u}^-)$ are invariant under Kashiwara's involution $\ast$, which implies part (1). Part (2) (resp., part (3)) is an immediate consequence of \cite[Lemma 7.3.2]{Kas2} (resp., \cite[equation (3.1.2)]{Kas4}). Now we deduce the first assertion of part (4) from part (2) and part (3) (here we take a dominant integral weight $\lambda$ such that $\langle\lambda, h_i\rangle$, $i \in I$, are sufficiently large for each $b \in \mathcal{B}(\infty)$). Finally, the second assertion of part (4) follows from the first assertion of part (4) and from the equalities 
\begin{align*}
(G_q ^{\rm low} (b) \cdot f_i ^{(k)})^\ast &= (f_i ^{(k)})^\ast \cdot G_q ^{\rm low} (b)^\ast\quad ({\rm since}\ \ast\ {\rm is\ a}\ \q(q)\mathchar`-{\rm algebra\ anti\mathchar`-involution})\\
&= f_i ^{(k)} \cdot G_q ^{\rm low} (b^\ast)\quad ({\rm by\ part}\ (1)).
\end{align*}
\end{proof}

\begin{defi}\normalfont
For $w \in W$ and a dominant integral weight $\lambda$, let $v_{q, w\lambda} \in V_q (\lambda)$ denote the (extremal) weight vector of weight $w \lambda$. Then, the $U_q (\mathfrak{u})$-submodule $V_{q, w} (\lambda) := U_q (\mathfrak{u}) \cdot v_{q, w\lambda} \subset V_q (\lambda)$ is called the {\it Demazure module} corresponding to $w$.
\end{defi}\vspace{2mm}

Denote by $b_\infty \in \mathcal{B}(\infty)$ the element corresponding to $1 \in U_q (\mathfrak{u}^-)$, and by $b_\lambda \in \mathcal{B}(\lambda)$ the highest weight element.

\vspace{2mm}\begin{prop}[{see \cite[Propositions 3.2.3, 3.2.5]{Kas4}}]\label{Demazure crystal}
Let ${\bf i} = (i_r, \ldots, i_1)$ be a reduced word for $w \in W$ $($notice the ordering of indices$)$, and $\lambda$ a dominant integral weight.
\begin{enumerate}
\item[{\rm (1)}] The subset \[\mathcal{B}_w(\infty) := \{\tilde{f}_{i_r} ^{a_r} \cdots \tilde{f}_{i_1} ^{a_1} b_\infty \mid a_1, \ldots, a_r \in \z_{\ge 0}\} \subset \mathcal{B}(\infty)\] is independent of the choice of a reduced word ${\bf i}$, and $\tilde{e}_i \mathcal{B}_w(\infty) \subset \mathcal{B}_w(\infty) \cup \{0\}$ for all $i \in I$.
\item[{\rm (2)}] The subset \[\mathcal{B}_w(\lambda) := \{\tilde{f}_{i_r} ^{a_r} \cdots \tilde{f}_{i_1} ^{a_1} b_\lambda \mid a_1, \ldots, a_r \in \z_{\ge 0}\} \setminus \{0\} \subset \mathcal{B}(\lambda)\] is independent of the choice of a reduced word ${\bf i}$, and $\tilde{e}_i \mathcal{B}_w(\lambda) \subset \mathcal{B}_w(\lambda) \cup \{0\}$ for all $i \in I$.
\item[{\rm (3)}] The equality $\pi_\lambda(\mathcal{B}_w(\infty)) = \mathcal{B}_w(\lambda) \cup \{0\}$ holds$;$ hence $\pi_\lambda$ induces a bijective map $\pi_\lambda: \widetilde{\mathcal{B}}_w(\lambda) \xrightarrow{\sim} \mathcal{B}_w(\lambda)$, where $\widetilde{\mathcal{B}}_w(\lambda) := \mathcal{B}_w(\infty) \cap \widetilde{\mathcal{B}}(\lambda)$.
\item[{\rm (4)}] The set $\{G^{\rm low} _{q, \lambda} (b) \mid b \in \mathcal{B}_w (\lambda)\}$ forms a $\q(q)$-basis of $V_{q, w} (\lambda)$.
\end{enumerate}
\end{prop}\vspace{2mm}

The subsets $\mathcal{B}_w (\infty), \mathcal{B}_w (\lambda)$ are called {\it Demazure crystals}.

\subsection{Polyhedral realizations of Demazure crystals}

In this subsection, we recall some fundamental properties of Nakashima-Zelevinsky polyhedral realizations of crystal bases (following \cite{NZ}, \cite{N1}, and \cite{N2}). Fix a reduced word ${\bf i} = (i_N, \ldots, i_1)$ for the longest element $w_0 \in W$, and consider an extension $\tilde{\bf i} = (\ldots, i_k, \ldots, i_{N+1}, i_N, \ldots, i_1)$ of ${\bf i}$ such that $i_k \neq i_{k+1}$ for all $k \ge 1$, and such that the cardinality of $\{k \ge 1 \mid i_k = i\}$ is $\infty$ for each $i \in I$. Following \cite{Kas4} and \cite{NZ}, we associate to $\tilde{\bf i}$ a crystal structure on $\z^{\infty} := \{(\ldots, a_k, \ldots, a_2, a_1) \mid a_k \in \z\ {\rm and}\ a_k = 0\ {\rm for}\ k \gg 0\}$ as follows. For $k \ge 1$, $i \in I$, and ${\bf a} = (\ldots, a_j, \ldots, a_2, a_1) \in \z^\infty$, we set 
\begin{align*}
\sigma_k({\bf a}) &:= a_k + \sum_{j > k} \langle \alpha_{i_j}, h_{i_k} \rangle a_j \in \z,\\
\sigma^{(i)}({\bf a}) &:= \max\{\sigma_k({\bf a}) \mid k \ge 1,\ i_k = i\} \in \z,\ {\rm and}\\
M^{(i)}({\bf a}) &:= \{k \ge 1 \mid i_k = i,\ \sigma_k({\bf a}) = \sigma^{(i)}({\bf a})\}.
\end{align*}
Since $a_j = 0$ for $j \gg 0$, the integers $\sigma_k({\bf a}), \sigma^{(i)}({\bf a})$ are well-defined; also, we have $\sigma^{(i)}({\bf a}) \ge 0$. Moreover, $M^{(i)}({\bf a})$ is a finite set if and only if $\sigma^{(i)}({\bf a}) > 0$. Define a crystal structure on $\z^\infty$ by 
\begin{align*}
{\rm wt}({\bf a}) &:= - \sum_{j = 1} ^\infty a_j \alpha_{i_j},\ \varepsilon_i({\bf a}) := \sigma^{(i)}({\bf a}),\ \varphi_i({\bf a}) := \varepsilon_i ({\bf a}) + \langle {\rm wt}({\bf a}), h_i \rangle,\ {\rm and}\\
\tilde{e}_i {\bf a} &:= 
\begin{cases}
(a_k - \delta_{k, \max M^{(i)}({\bf a})})_{k \ge 1} &{\rm if}\ \sigma^{(i)}({\bf a}) > 0,\\
0 &{\rm otherwise},
\end{cases}\\
\tilde{f}_i {\bf a} &:= (a_k + \delta_{k, \min M^{(i)}({\bf a})})_{k \ge 1} 
\end{align*}
for $i \in I$ and ${\bf a} = (\ldots, a_k, \ldots, a_2, a_1) \in \z^\infty$; we denote this crystal by $\z^\infty _{\tilde{\bf i}}$.

\vspace{2mm}\begin{prop}[{see \cite[\S\S 2.4]{NZ}}]\label{star string}
The following hold.
\begin{enumerate}
\item[{\rm (1)}] There exists a unique strict embedding of crystals $\Psi_{\tilde{\bf i}}: \mathcal{B}(\infty) \hookrightarrow \z^\infty _{\tilde{\bf i}}$ such that $\Psi_{\tilde{\bf i}} (b_\infty) = (\ldots, 0, \ldots, 0, 0)$.
\item[{\rm (2)}] For $b \in \mathcal{B}(\infty)$, write $\Psi_{\tilde{\bf i}} (b) = (\ldots, a_k, \ldots, a_2, a_1)$. Then, \[b^\ast = \tilde{f}_{i_1} ^{a_1} \tilde{f}_{i_2} ^{a_2} \cdots b_\infty,\ {\it and}\ \tilde{e}_{i_{k-1}}\tilde{f}_{i_k} ^{a_k} \tilde{f}_{i_{k+1}} ^{a_{k+1}} \cdots b_\infty = 0\ {\it for\ all}\ k >1;\] namely, the sequence $(a_k)_{k \ge 1}$ is identical to the string parameterization of $b^\ast$ with respect to the sequence $(i_1, i_2, \ldots, i_k, \ldots)$. 
\end{enumerate}
\end{prop}\vspace{2mm}

The embedding $\Psi_{\tilde{\bf i}}$ (resp., the image $\Psi_{\tilde{\bf i}}(\mathcal{B}(\infty))$) is called the {\it Kashiwara embedding} (resp., the {\it Nakashima-Zelevinsky polyhedral realization} of $\mathcal{B}(\infty)$) with respect to $\tilde{\bf i}$. 

\vspace{2mm}\begin{rem}\normalfont\label{definition1}
By \cite[Proposition 3.2.5 (ii)]{Kas4}, it follows that $\tilde{f}_i \mathcal{B}_{w_0}(\infty) \subset \mathcal{B}_{w_0}(\infty)$ for all $i \in I$, and hence that $\mathcal{B}_{w_0}(\infty) = \mathcal{B}(\infty)$. Since ${\bf i} = (i_N, \ldots, i_1)$ is a reduced word for $w_0$, we deduce that $\Psi_{\tilde{\bf i}}(\mathcal{B}(\infty)) \subset \{(\ldots, a_k, \ldots, a_2, a_1) \in \z^\infty _{\tilde{\bf i}} \mid a_k = 0\ {\rm for\ all}\ k >N\}$; hence the image $\Psi_{\tilde{\bf i}}(\mathcal{B}(\infty))$ can be regarded as a subset of $\z^N$ by \[\Psi_{\tilde{\bf i}}(\mathcal{B}(\infty)) \hookrightarrow \z^N,\ (\ldots, 0, 0, a_N, \ldots, a_2, a_1) \mapsto (a_1, a_2, \ldots, a_N);\] note that the ordering of entries is reversed. Since the composite map $\mathcal{B}(\infty) \xrightarrow{\Psi_{\tilde{\bf i}}} \Psi_{\tilde{\bf i}}(\mathcal{B}(\infty)) \hookrightarrow \z^N$ is independent of the choice of an extension $\tilde{\bf i}$ by Proposition \ref{star string} (2), we denote this map simply by $\Psi_{\bf i}$; this is also called the {\it Kashiwara embedding} with respect to ${\bf i}$. Here we should mention that Nakashima-Zelevinsky (\cite{NZ}) called an explicit description $\Sigma_{\tilde{\bf i}}$ of $\Psi_{\tilde{\bf i}}(\mathcal{B}(\infty))$ a polyhedral realization (see \cite[Theorem 3.1]{NZ}); our terminology is slightly different from the original one.
\end{rem}\vspace{2mm}

Let us write $\z^\infty _{\tilde{\bf i}} [\lambda] := \z^\infty _{\tilde{\bf i}} \otimes R_\lambda$, and identify $\z^\infty _{\tilde{\bf i}} [\lambda]$ with $\z^\infty _{\tilde{\bf i}}$ as a set; note that their crystal structures are different. By \cite[Theorem 3.1]{N1}, there exists a unique strict embedding of crystals \[\Omega_\lambda: \mathcal{B} (\lambda) \hookrightarrow \mathcal{B}(\infty) \otimes R_\lambda\] such that $\Omega_\lambda(b_\lambda) = b_\infty \otimes r_\lambda$. Remark that $\Omega_\lambda (\mathcal{B}(\lambda)) = \{b \otimes r_\lambda \mid b \in \widetilde{\mathcal{B}}(\lambda)\}$, and that $\Omega_\lambda(\pi_\lambda (b)) = b \otimes r_\lambda$ for all $b \in \widetilde{\mathcal{B}}(\lambda)$, where $\pi_\lambda: \widetilde{\mathcal{B}}(\lambda) \xrightarrow{\sim} \mathcal{B}(\lambda)$ denotes the bijective map given in Proposition \ref{connection of crystals} (2). Next we consider an arbitrary element $w \in W$. Let ${\bf i} = (i_r, \ldots, i_1)$ be a reduced word for $w$, and extend it to a reduced word $\hat{\bf i} = (i_N, \ldots, i_{r+1}, i_r, \ldots, i_1)$ for $w_0$. In addition, we take an extension $\tilde{\bf i} = (\ldots, i_k, \ldots, i_{N+1}, i_N, \ldots, i_1)$ of $\hat{\bf i}$ as above.

\vspace{2mm}\begin{prop}[{\cite[Theorem 3.2]{N1} and \cite[Propositions 3.1, 3.3 (i)]{N2}}]\label{polyhedral realizations}
For a dominant integral weight $\lambda$, the following hold.
\begin{enumerate}
\item[{\rm (1)}] There exists a unique strict embedding of crystals \[\Psi^{(\lambda)} _{\tilde{\bf i}}: \mathcal{B} (\lambda) \lhook\joinrel\xrightarrow{\Omega_\lambda} \mathcal{B}(\infty) \otimes R_\lambda \lhook\joinrel\xrightarrow{\Psi_{\tilde{\bf i}} \otimes {\rm id}} \z^\infty _{\tilde{\bf i}} \otimes R_\lambda = \z^\infty _{\tilde{\bf i}} [\lambda]\] such that $\Psi^{(\lambda)} _{\tilde{\bf i}} (b_\lambda) = (\ldots, 0, \ldots, 0, 0) \otimes r_\lambda$.
\item[{\rm (2)}] There hold the equalities
\begin{align*}
&\Psi_{\tilde{\bf i}} (\mathcal{B}_w (\infty)) = \{(\ldots, a_k, \ldots, a_2, a_1) \in \Psi_{\tilde{\bf i}} (\mathcal{B}(\infty)) \mid a_k = 0\ {\it for\ all}\ k >r\},\\
&\Psi^{(\lambda)} _{\tilde{\bf i}}(\mathcal{B}_w (\lambda)) = \{(\ldots, a_k, \ldots, a_2, a_1) \otimes r_\lambda \in \Psi^{(\lambda)} _{\tilde{\bf i}} (\mathcal{B} (\lambda)) \mid a_k = 0\ {\it for\ all}\ k >r\}.
\end{align*}
\end{enumerate} 
\end{prop}\vspace{2mm}

We call $\Psi^{(\lambda)} _{\tilde{\bf i}}(\mathcal{B}_w (\lambda))$ the {\it polyhedral realization} of $\mathcal{B}_w (\lambda)$ with respect to $\tilde{\bf i}$. 

\vspace{2mm}\begin{defi}\normalfont\label{definition2}
The sets $\Psi_{\tilde{\bf i}} (\mathcal{B}_w (\infty))$ and $\Psi^{(\lambda)} _{\tilde{\bf i}}(\mathcal{B}_w (\lambda))$ can be regarded as subsets of $\z^r$ by 
\begin{align*}
\Psi_{\tilde{\bf i}} (\mathcal{B}_w (\infty)) &\hookrightarrow \z^r,\ (\ldots, 0, 0, a_r, \ldots, a_2, a_1) \mapsto (a_1, a_2, \ldots, a_r),\ {\rm and}\\ 
\Psi^{(\lambda)} _{\tilde{\bf i}}(\mathcal{B}_w (\lambda)) &\hookrightarrow \z^r,\ (\ldots, 0, 0, a_r, \ldots, a_2, a_1) \otimes r_\lambda \mapsto (a_1, a_2, \ldots, a_r).
\end{align*}
Since the composite maps $\mathcal{B}_w (\infty) \xrightarrow{\Psi_{\tilde{\bf i}}} \Psi_{\tilde{\bf i}}(\mathcal{B}_w (\infty)) \hookrightarrow \z^r$ and $\mathcal{B}_w (\lambda) \xrightarrow{\Psi_{\tilde{\bf i}} ^{(\lambda)}} \Psi^{(\lambda)} _{\tilde{\bf i}}(\mathcal{B}_w (\lambda)) \hookrightarrow \z^r$ are independent of the choices of $\hat{\bf i}$ and $\tilde{\bf i}$, we denote these maps simply by $\Psi_{\bf i} ^{(w)}$ and $\Psi_{\bf i} ^{(\lambda, w)}$, respectively; these are also called the {\it Kashiwara embeddings} with respect to ${\bf i}$.
\end{defi} 

\vspace{2mm}\begin{defi}\normalfont\label{definition3}
Define a subset $\mathcal{S}_{\bf i} ^{(\lambda, w)} \subset \z_{>0} \times \z^r$ by \[\mathcal{S}_{\bf i} ^{(\lambda, w)} := \bigcup_{k >0} \{(k, \Psi_{\bf i} ^{(k\lambda, w)}(b)) \mid b \in \mathcal{B}_w (k\lambda)\},\] and denote by $\mathcal{C}_{\bf i} ^{(\lambda, w)} \subset \r_{\ge 0} \times \r^r$ the smallest real closed cone containing $\mathcal{S}_{\bf i} ^{(\lambda, w)}$, that is, \[\mathcal{C}_{\bf i} ^{(\lambda, w)}:= \overline{\{c \cdot (k, {\bf a}) \mid c \in \r_{>0}\ {\rm and}\ (k, {\bf a}) \in \mathcal{S}_{\bf i} ^{(\lambda, w)}\}},\] where $\overline{H}$ means the closure of $H \subset \r_{\ge 0} \times \r^r$ with respect to the Euclidean topology. Now let us define a subset $\Delta_{\bf i} ^{(\lambda, w)} \subset \r^r$ by \[\Delta_{\bf i} ^{(\lambda, w)} := \{{\bf a} \in \r^r \mid (1, {\bf a}) \in \mathcal{C}_{\bf i} ^{(\lambda, w)}\}.\] The set $\Delta_{\bf i} ^{(\lambda, w)}$ is also called the {\it polyhedral realization} of $\mathcal{B}_w(\lambda)$ with respect to ${\bf i}$. 
\end{defi}\vspace{2mm}

We will prove that the set $\mathcal{S}_{\bf i} ^{(\lambda, w)}$ is identical to the set of all integral points in $\mathcal{C}_{\bf i} ^{(\lambda, w)} \setminus \{0\}$, and that the polyhedral realization $\Psi_{\bf i} ^{(\lambda, w)}(\mathcal{B}_w(\lambda))$ of $\mathcal{B}_w(\lambda)$ is identical to the set of all integral points in $\Delta_{\bf i} ^{(\lambda, w)}$. We need the following lemma.

\vspace{2mm}\begin{lem}[{see \cite[Theorem 3.1]{N1}}]\label{piecewise-linear inequality}
For the subset $\widetilde{\mathcal{B}} (\lambda)$ of $\mathcal{B}(\infty)$ defined in Proposition \ref{connection of crystals} (2), there holds the equality \[\widetilde{\mathcal{B}} (\lambda) = \{b \in \mathcal{B}(\infty) \mid \varepsilon_i ^\ast(b) \le \langle\lambda, h_i\rangle\ {\it for\ all}\ i \in I\}.\]
\end{lem}\vspace{2mm}

For ${\bf i}^{\rm op} := (i_1, \ldots, i_r)$, which is a reduced word for $w^{-1}$, let $\Phi_{{\bf i}^{\rm op}}: \mathcal{B}_{w^{-1}}(\infty) \hookrightarrow \z^r$ denote the corresponding string parameterization; namely, if we write $\Phi_{{\bf i}^{\rm op}} (b) = (a_1, \ldots, a_r)$ for $b \in \mathcal{B}_{w^{-1}} (\infty)$, then we have \[b = \tilde{f}_{i_1} ^{a_1} \tilde{f}_{i_2} ^{a_2} \cdots \tilde{f}_{i_r} ^{a_r} b_\infty,\ {\rm and}\ \tilde{e}_{i_{k-1}}\tilde{f}_{i_k} ^{a_k} \tilde{f}_{i_{k+1}} ^{a_{k+1}} \cdots \tilde{f}_{i_r} ^{a_r} b_\infty = 0\ {\rm for}\ 1 < k \le r.\] It follows from Proposition \ref{star string} (2) that $\Psi_{\bf i} ^{(w)} = \Phi_{{\bf i}^{\rm op}} \circ \ast$ on $\mathcal{B}_w(\infty)$; here we note that $\mathcal{B}_w(\infty)^\ast = \mathcal{B}_{w^{-1}}(\infty)$ by \cite[Proposition 3.3.1]{Kas4}. Therefore, if we denote by $\mathcal{C}_{\bf i} \subset \r^r$ the smallest real closed cone containing $\Psi_{\bf i} ^{(w)}(\mathcal{B}_w(\infty))$, then $\mathcal{C}_{\bf i}$ is identical to the {\it string cone} associated to ${\bf i}^{\rm op}$ (see \cite[\S\S 3.2]{BZ}), and the equality $\Psi_{\bf i} ^{(w)}(\mathcal{B}_w(\infty)) = \mathcal{C}_{\bf i} \cap \z^r$ holds (see \cite[Section 1]{Lit} and \cite[Proposition 3.5]{BZ}). Note that a system of explicit linear inequalities defining $\mathcal{C}_{\bf i}$ is given in \cite[Theorem 3.10]{BZ}. In particular, the real closed cone $\mathcal{C}_{\bf i}$ is a rational convex polyhedral cone, that is, there exists a finite number of rational points ${\bf a}_1, \ldots, {\bf a}_l \in \q^r$ such that $\mathcal{C}_{\bf i} = \r_{\ge 0}{\bf a}_1 + \cdots + \r_{\ge 0} {\bf a}_l$. Now for $(a_1, \ldots, a_r) \in \Psi_{\bf i} ^{(w)}(\mathcal{B}_w(\infty))$, Proposition \ref{connection of crystals} (3) implies that $\pi_\lambda(\tilde{f}_{i_1} ^{a_1} \cdots \tilde{f}_{i_r} ^{a_r} b_\infty) = \tilde{f}_{i_1} ^{a_1} \cdots \tilde{f}_{i_r} ^{a_r} b_\lambda$. Therefore, if we take a dominant integral weight $\lambda$ such that $\tilde{f}_{i_1} ^{a_1} \cdots \tilde{f}_{i_r} ^{a_r} b_\lambda \neq 0$, then we deduce from the remark following Proposition \ref{connection of crystals} that \[\varepsilon_i (\tilde{f}_{i_1} ^{a_1} \cdots \tilde{f}_{i_r} ^{a_r} b_\infty) = \varepsilon_i (\tilde{f}_{i_1} ^{a_1} \cdots \tilde{f}_{i_r} ^{a_r} b_\lambda)\] for all $i \in I$. Using this observation, we obtain the following.

\vspace{2mm}\begin{prop}[{see \cite[Remark 5.4 and Corollary 5.20]{F}}]\label{piecewise-linear for epsilon}
There exists a piecewise-linear function $\psi_{\bf i} ^{(i)}$ of $a_1, \ldots, a_r$ for $i \in I$ such that $\varepsilon_i (\tilde{f}_{i_1} ^{a_1} \cdots \tilde{f}_{i_r} ^{a_r} b_\infty) = \psi_{\bf i} ^{(i)} (a_1, \ldots, a_r)$ for all $(a_1, \ldots, a_r) \in \Psi_{\bf i} ^{(w)}(\mathcal{B}_w(\infty))$.
\end{prop}\vspace{2mm}

\begin{cor}\label{a finite union}
The following hold.
\begin{enumerate}
\item[{\rm (1)}] The set $\mathcal{S}_{\bf i} ^{(\lambda, w)}$ is identical to the set of $(k, a_1, \ldots, a_r) \in \z_{>0} \times \z^r$ such that $(a_1, \ldots, a_r) \in \mathcal{C}_{\bf i}$, and such that $\psi_{\bf i} ^{(i)} (a_1, \ldots, a_r) \le \langle k\lambda, h_i\rangle$ for all $i \in I$. In particular, the real closed cone $\mathcal{C}_{\bf i} ^{(\lambda, w)}$ is identical to the set of $(k, a_1, \ldots, a_r) \in \r_{\ge 0} \times \mathcal{C}_{\bf i}$ such that $\psi_{\bf i} ^{(i)} (a_1, \ldots, a_r) \le \langle k\lambda, h_i\rangle$ for all $i \in I$.
\item[{\rm (2)}] The real closed cone $\mathcal{C}_{\bf i} ^{(\lambda, w)}$ is a finite union of rational convex polyhedral cones, and the equality $\mathcal{S}_{\bf i} ^{(\lambda, w)} = \mathcal{C}_{\bf i} ^{(\lambda, w)} \cap (\z_{>0} \times \z^r)$ holds.
\end{enumerate}
\end{cor}

\begin{proof}
Part (2) is an immediate consequence of part (1); hence it is sufficient to prove part (1). By Proposition \ref{Demazure crystal} (3) and Lemma \ref{piecewise-linear inequality}, we deduce that 
\begin{align*}
\Psi_{\bf i} ^{(k\lambda, w)} (\mathcal{B}_w (k\lambda)) &= \{(a_1, \ldots, a_r) \in \Psi_{\bf i} ^{(w)}(\mathcal{B}_w(\infty)) \mid \varepsilon_i (\tilde{f}_{i_1} ^{a_1} \cdots \tilde{f}_{i_r} ^{a_r} b_\infty) \le \langle k\lambda, h_i\rangle\ {\rm for\ all}\ i \in I\}\\
&= \{(a_1, \ldots, a_r) \in \mathcal{C}_{\bf i} \cap \z^r \mid \psi_{\bf i} ^{(i)} (a_1, \ldots, a_r) \le \langle k\lambda, h_i\rangle\ {\rm for\ all}\ i \in I\}\\
&({\rm by\ Proposition}\ \ref{piecewise-linear for epsilon}\ {\rm since}\ \Psi_{\bf i} ^{(w)}(\mathcal{B}_w(\infty)) = \mathcal{C}_{\bf i} \cap \z^r)
\end{align*}
for all $k \in \z_{>0}$. This implies the first assertion of part (1). Then, since $\psi_{\bf i} ^{(i)}$ is piecewise-linear, the second assertion of part (1) follows immediately. This proves the corollary. 
\end{proof}

By the definition of $\Delta_{\bf i} ^{(\lambda, w)}$, we obtain the following from Corollary \ref{a finite union} (1).

\vspace{2mm}\begin{cor}\label{piecewise-linear inequality for delta}
The polyhedral realization $\Delta_{\bf i} ^{(\lambda, w)}$ is identical to the set of $(a_1, \ldots, a_r) \in \mathcal{C}_{\bf i}$ such that $\psi_{\bf i} ^{(i)} (a_1, \ldots, a_r) \le \langle \lambda, h_i\rangle$ for all $i \in I$.
\end{cor}\vspace{2mm}

\begin{cor}\label{a finite union of polytopes}
The polyhedral realization $\Delta_{\bf i} ^{(\lambda, w)}$ is a finite union of rational convex polytopes, and the equality $\Delta_{\bf i} ^{(\lambda, w)} \cap \z^r = \Psi_{\bf i} ^{(\lambda, w)}(\mathcal{B}_w(\lambda))$ holds$;$ here, a subset $\Delta \subset \r^r$ is called a rational convex polytope if it is a convex hull of a finite number of rational points.
\end{cor}

\begin{proof}
By Lemma \ref{piecewise-linear inequality}, we deduce that $0 \le a_k \le \langle \lambda, h_{i_k} \rangle$, $1 \le k \le r$, for all $(a_1, \ldots, a_r) \in \Delta_{\bf i} ^{(\lambda, w)}$; therefore, the polyhedral realization $\Delta_{\bf i} ^{(\lambda, w)}$ is bounded, and hence compact. Also, by Corollary \ref{piecewise-linear inequality for delta}, the polyhedral realization $\Delta_{\bf i} ^{(\lambda, w)}$ is given by a finite number of piecewise-linear inequalities. These imply the first assertion of the corollary. Now the second assertion is an immediate consequence of Corollary \ref{a finite union} (2). This proves the corollary.
\end{proof}

In Section 4, we will prove that the polyhedral realization $\Delta_{\bf i} ^{(\lambda, w)}$ is identical to the Newton-Okounkov convex body of a Schubert variety with respect to a specific valuation, and that these are indeed rational convex polytopes. Note that we do not necessarily assume that $(\tilde{\bf i}, \lambda)$ is ample (see Definition \ref{definition of ample}); when it is ample, we obtain a system of explicit affine inequalities defining the polyhedral realization $\Delta_{\bf i} ^{(\lambda, w)}$ (see Corollary \ref{explicit description}).

\section{Newton-Okounkov convex bodies of Schubert varieties}
\subsection{Newton-Okounkov convex bodies}

Here we review the definition of Newton-Okounkov convex bodies (following \cite{HK}, \cite{Kav}, \cite{KK1}, and \cite{KK2}). Let $R$ be a $\c$-algebra without nonzero zero-divisors, and fix a total order $<$ on $\z^r$, $r \ge 1$, respecting the addition. 

\vspace{2mm}\begin{defi}\normalfont\label{def,val}
A map $v: R \setminus \{0\} \rightarrow \z^r$ is called a {\it valuation} on $R$ if the following hold: for every $\sigma, \tau \in R \setminus \{0\}$ and $c \in \c \setminus \{0\}$,
\begin{enumerate}
\item[{\rm (i)}]  $v(\sigma \cdot \tau) = v(\sigma) + v(\tau)$,
\item[{\rm (ii)}] $v(c \cdot \sigma) = v(\sigma)$, 
\item[{\rm (iii)}] $v (\sigma + \tau) \ge {\rm min} \{v(\sigma), v(\tau) \}$ unless $\sigma + \tau = 0$. 
\end{enumerate}
\end{defi}\vspace{2mm}

Note that we need to fix a total order on $\z^r$ whenever we consider a valuation. The following is a fundamental property of valuations.

\vspace{2mm}\begin{prop}[{see, for instance, \cite[Proposition 1.8]{Kav}}]\label{prop1,val}
Let $v$ be a valuation on $R$. Assume that $\sigma_1, \ldots, \sigma_s \in R \setminus \{0\}$, and that $v(\sigma_1), \ldots, v(\sigma_s)$ are all distinct.
\begin{enumerate}
\item[{\rm (1)}] The elements $\sigma_1, \ldots, \sigma_s$ are linearly independent over $\c$.
\item[{\rm (2)}] For $c_1, \ldots, c_s \in \c$ such that $\sigma := c_1 \sigma_1 + \cdots + c_s \sigma_s \neq 0$, \[v(\sigma) = \min\{v(\sigma_t ) \mid 1 \le t \le s,\ c_t \neq 0 \}.\]
\end{enumerate}
\end{prop}\vspace{2mm}

For ${\bf a} \in \z^r$ and a valuation $v$ on $R$ with values in $\z^r$, we set $R_{\bf a} := \{\sigma \in R\ |\ \sigma = 0\ {\rm or}\ v(\sigma) \ge {\bf a}\}$; this is a $\c$-subspace of $R$. The {\it leaf} above ${\bf a} \in \z^r$ is defined to be the quotient space $\widehat{R}_{\bf a} := R_{\bf a}/\bigcup_{{\bf a}<{\bf b}} R_{\bf b}$. A valuation $v$ is said to have {\it one-dimensional leaves} if dim$(\widehat{R}_{\bf a}) = 0\ {\rm or}\ 1$ for all ${\bf a} \in \z^r$. 

\vspace{2mm}\begin{ex}\normalfont\label{highest term valuation}
Define two lexicographic orders $<$ and $\prec$ on $\z^r$ by $(a_1, \ldots, a_r) < (a_1 ^\prime, \ldots, a_r ^\prime)$ (resp., $(a_1, \ldots, a_r) \prec (a_1 ^\prime, \ldots, a_r ^\prime)$) if and only if there exists $1 \le k \le r$ such that $a_1 = a_1 ^\prime, \ldots, a_{k-1} = a_{k-1} ^\prime$, $a_k < a_k ^\prime$ (resp., $a_r = a_r ^\prime, \ldots, a_{k+1} = a_{k+1} ^\prime$, $a_k < a_k ^\prime$). Let $\c(t_r, \ldots, t_1)$ denote the rational function field in $r$ variables. The lexicographic order $<$ (resp., $\prec$) on $\z^r$ induces a total order (denoted by the same symbol $<$ (resp., $\prec$)) on the set of all monomials in the polynomial ring $\c[t_r, \ldots, t_1]$ as follows: $t_r ^{a_r} \cdots t_1 ^{a_1} < t_r ^{a_r ^\prime} \cdots t_1 ^{a_1 ^\prime}$ (resp., $t_r ^{a_r} \cdots t_1 ^{a_1} \prec t_r ^{a_r ^\prime} \cdots t_1 ^{a_1 ^\prime}$) if and only if $(a_1, \ldots, a_r) < (a_1 ^\prime, \ldots, a_r ^\prime)$ (resp., $(a_1, \ldots, a_r) \prec (a_1 ^\prime, \ldots, a_r ^\prime)$). First we consider the case that the fixed total order on $\z^r$ is the lexicographic order $<$. Let us define a map $v: \c(t_r, \ldots, t_1) \setminus \{0\} \rightarrow \z^r$ by $v(f/g) := v(f) - v(g)$ for $f, g \in \c[t_r, \ldots, t_1] \setminus \{0\}$, and by \[v(f) := -(a_1, \ldots, a_r)\ {\rm for}\ f = c t_r ^{a_r} \cdots t_1 ^{a_1} + ({\rm lower\ terms}) \in \c[t_r, \ldots, t_1] \setminus \{0\},\] where $c \in \c \setminus \{0\}$, and by ``lower terms'' we mean a linear combination of monomials smaller than $t_r ^{a_r} \cdots t_1 ^{a_1}$ in the total order $<$. It is obvious that $v$ is a valuation with one-dimensional leaves. Since the induced order $<$ on the set of all monomials satisfies $t_1 > \cdots > t_r$, we call the valuation $v$ on $\c(t_r, \ldots, t_1)$ the {\it highest term valuation with respect to the lexicographic order} $t_1 > \cdots > t_r$. Similarly, the {\it highest term valuation} $\tilde{v}$ {\it with respect to the lexicographic order} $t_r \succ \cdots \succ t_1$ is defined to be the valuation on $\c(t_r, \ldots, t_1)$ given by \[\tilde{v}(f) := -(a_r, \ldots, a_1)\ {\rm for}\ f = c t_r ^{a_r} \cdots t_1 ^{a_1} + ({\rm lower\ terms}) \in \c[t_r, \ldots, t_1] \setminus \{0\},\] where $c \in \c \setminus \{0\}$, and by ``lower terms'' we mean a linear combination of monomials smaller than $t_r ^{a_r} \cdots t_1 ^{a_1}$ in the total order $\prec$; note that the induced order $\prec$ on the set of all monomials satisfies $t_r \succ \cdots \succ t_1$. If $r = 3$ and $f = t_1 t_2 + t_3 ^2 \in \c[t_3, t_2, t_1]$, then we have $v(f) = -(1, 1, 0)$ and $\tilde{v}(f) = -(2, 0, 0)$.
\end{ex}\vspace{2mm}

Let $V(\lambda)$ be the irreducible highest weight $G$-module with highest weight $\lambda$, and $v_\lambda \in V(\lambda)$ the highest weight vector. Given a dominant integral weight $\lambda$, we define a line bundle $\mathcal{L}_\lambda$ on $G/B$ by \[\mathcal{L}_\lambda := (G \times \c)/B,\] where $B$ acts on $G \times \c$ on the right as follows: \[(g, c) \cdot b = (g b, \lambda(b) c)\] for $g \in G$, $c \in \c$, and $b \in B$. If we define a morphism $\rho_\lambda: G/B \rightarrow \mathbb{P}(V(\lambda))$ by \[g \bmod B \mapsto \c g \cdot v_\lambda,\] then we obtain $\rho_\lambda ^\ast (\mathcal{O}(1)) = \mathcal{L}_\lambda$, where $\mathcal{O}(1)$ denotes the twisting sheaf of Serre. Hence the morphism $\rho_\lambda$ induces a $\c$-linear map \[\rho_\lambda ^\ast: H^0(\mathbb{P}(V(\lambda)), \mathcal{O}(1)) \rightarrow H^0(G/B, \mathcal{L}_\lambda).\] From the Borel-Weil theorem, we know that this is an isomorphism of $G$-modules. Here we recall that for an arbitrary finite-dimensional $G$-module $V$ over $\c$, the space $H^0(\mathbb{P}(V), \mathcal{O}(1))$ of global sections is identified with the dual $G$-module $V^\ast :=\ {\rm Hom}_\c (V, \c)$. Hence the isomorphism $\rho_\lambda ^\ast$ is regarded as an isomorphism of $G$-modules from $V(\lambda)^\ast$ to $H^0(G/B, \mathcal{L}_\lambda)$.  

\vspace{2mm}\begin{defi}\normalfont
For $w \in W$, let us denote by $X(w)$ the Zariski closure of $B \widetilde{w} B/B$ in $G/B$, where $\widetilde{w} \in G$ denotes a lift for $w$; note that the closed subvariety $X(w)$ is independent of the choice of $\widetilde{w}$. The $X(w)$ is called the {\it Schubert variety} corresponding to $w \in W$.
\end{defi}\vspace{2mm}

It is well-known that $X(w)$ is a normal projective variety of complex dimension $\ell(w)$; here, $\ell(w)$ denotes the length of $w$. Now the line bundle $\mathcal{L}_\lambda$ on $G/B$ induces a line bundle on $X(w)$, which we denote by the same symbol $\mathcal{L}_\lambda$. 

\vspace{2mm}\begin{defi}\normalfont
For $w \in W$ and a dominant integral weight $\lambda$, let $v_{w\lambda} \in V(\lambda)$ denote the (extremal) weight vector of weight $w\lambda$. Define a $B$-submodule $V_w(\lambda) \subset V(\lambda)$ by \[V_w(\lambda) := \sum_{b \in B} \c b v_{w\lambda};\] this is called the {\it Demazure module} corresponding to $w \in W$.
\end{defi}\vspace{2mm}

From the Borel-Weil theorem, we know that the space $H^0(X(w), \mathcal{L}_\lambda)$ of global sections is a $B$-module isomorphic to the dual module $V_w (\lambda)^\ast$. The {\it ring of sections} $R(\mathcal{L}_\lambda)$ is the $\z_{\ge 0}$-graded $\c$-algebra obtained from $\mathcal{L}_\lambda$ by \[R(\mathcal{L}_\lambda) := \bigoplus_{k \ge 0} H^0(X(w), \mathcal{L}_\lambda ^{\otimes k}).\] If we take a section $\tau \in H^0(X(w), \mathcal{L}_\lambda) \setminus \{0\}$, then the $\c$-vector space $R(\mathcal{L}_\lambda)_k := H^0(X(w), \mathcal{L}_\lambda ^{\otimes k})$ can be regarded as a finite-dimensional subspace of $\c(X(w))$ as follows: \[R(\mathcal{L}_\lambda)_k \hookrightarrow \c(X(w)),\ \sigma \mapsto \sigma/\tau^k.\] Hence a valuation on $\c(X(w))$ induces a map from $R(\mathcal{L}_\lambda)_k \setminus \{0\}$. 

\vspace{2mm}\begin{defi}\normalfont\label{Newton-Okounkov convex body}
For $w \in W$ and a dominant integral weight $\lambda$, let $v: \c(X(w)) \setminus\{0\} \rightarrow \z^r$ be a valuation with one-dimensional leaves with $r := \ell(w)$, and $\tau \in H^0(X(w), \mathcal{L}_\lambda) \setminus \{0\}$. Define a subset $S(X(w), \mathcal{L}_\lambda, v, \tau) \subset \z_{>0} \times \z^r$ by \[S(X(w), \mathcal{L}_\lambda, v, \tau) := \bigcup_{k>0} \{(k, v(\sigma / \tau^k)) \mid \sigma \in R(\mathcal{L}_\lambda)_k \setminus \{0\}\},\] and denote by $C(X(w), \mathcal{L}_\lambda, v, \tau) \subset \r_{\ge 0} \times \r^r$ the smallest real closed cone containing $S(X(w), \mathcal{L}_\lambda, v, \tau)$. Let us define a subset $\Delta(X(w), \mathcal{L}_\lambda, v, \tau) \subset \r^r$ by \[\Delta(X(w), \mathcal{L}_\lambda, v, \tau) := \{{\bf a} \in \r^r \mid (1, {\bf a}) \in C(X(w), \mathcal{L}_\lambda, v, \tau)\};\]
this is called the {\it Newton-Okounkov convex body} of $X(w)$ associated to $\mathcal{L}_\lambda$, $v$, and $\tau$.
\end{defi}\vspace{2mm}

From the definition of valuations, it is obvious that the set $S(X(w), \mathcal{L}_\lambda, v, \tau)$ is a semigroup. Therefore, we deduce that the real closed cone $C(X(w), \mathcal{L}_\lambda, v, \tau)$ is a closed convex cone, and that the set $\Delta(X(w), \mathcal{L}_\lambda, v, \tau)$ is a convex set. Moreover, it follows from \cite[Theorem 2.30]{KK2} that $\Delta(X(w), \mathcal{L}_\lambda, v, \tau)$ is a convex body, i.e., a compact convex set. It is well-known that $\mathcal{L}_\lambda$ is very ample if and only if $\lambda$ is a regular dominant integral weight, i.e., $\langle \lambda, h_i\rangle \in \z_{>0}$ for all $i \in I$. In this case, the real dimension of $\Delta(X(w), \mathcal{L}_\lambda, v, \tau)$ is equal to $r$ ($= \ell(w)$) by \cite[Corollary 3.2]{KK2}; this is not necessarily the case if $\lambda$ is not regular.

\vspace{2mm}\begin{rem}\normalfont
If $\mathcal{L}_\lambda$ is a very ample line bundle, then we can take a closed immersion $X(w) \hookrightarrow \mathbb{P}(H^0(X(w), \mathcal{L}_\lambda)^\ast)$ such that $\mathcal{L}_\lambda$ is the pullback of the twisting sheaf $\mathcal{O}(1)$ of Serre. Denote by $R = \bigoplus_{k \ge 0} R_k$ the corresponding homogeneous coordinate ring. In many literatures including \cite{HK}, Newton-Okounkov convex bodies are defined by using $R$ instead of $R(\mathcal{L}_\lambda)$. However, since the Schubert variety $X(w)$ is normal, we deduce from \cite[Chap.\ I\hspace{-.1em}I,  Ex.\ 5.14]{Hart} that $R_k = R(\mathcal{L}_\lambda)_k$ for all $k \gg 0$. In addition, since $S(X(w), \mathcal{L}_\lambda, v, \tau)$ is a semigroup, the real closed cone $C(X(w), \mathcal{L}_\lambda, v, \tau)$ is identical to the smallest real closed cone containing \[\bigcup_{k>k^\prime} \{(k, v(\sigma / \tau^k)) \mid \sigma \in R(\mathcal{L}_\lambda)_k \setminus \{0\}\}\] for $k^\prime \gg 0$. Therefore, $R$ and $R(\mathcal{L}_\lambda)$ are interchangeable in the definition of Newton-Okounkov convex bodies.
\end{rem}\vspace{2mm}

\begin{rem}\normalfont\label{independence}
If we take another section $\tau^\prime \in H^0 (X(w), \mathcal{L}_\lambda) \setminus \{0\}$, then $S(X(w), \mathcal{L}_\lambda, v, \tau^\prime)$ is the shift of $S(X(w), \mathcal{L}_\lambda, v, \tau)$ by $k v(\tau/\tau^\prime)$ in $\{k\} \times \z^r$. Hence it follows that $\Delta(X(w), \mathcal{L}_\lambda, v, \tau^\prime) = \Delta(X(w), \mathcal{L}_\lambda, v, \tau) + v(\tau/\tau^\prime)$. Thus, the Newton-Okounkov convex body $\Delta(X(w), \mathcal{L}_\lambda, v, \tau)$ does not essentially depend on the choice of $\tau \in H^0 (X(w), \mathcal{L}_\lambda) \setminus \{0\}$.
\end{rem}

\subsection{Highest term valuations}

In this subsection, we introduce a specific valuation on $\c(X(w))$, which we mainly use. Let $P_i$ (resp., $U^- _i$) denote the minimal parabolic subgroup (resp., the opposite root subgroup) corresponding to an index $i \in I$, and set $\mathfrak{u}^- _i :=\ {\rm Lie}(U^- _i)$. For a reduced word ${\bf i} = (i_r, \ldots, i_1)$ for $w \in W$, define the {\it Bott-Samelson variety} $Z_{\bf i}$ by \[Z_{\bf i} := P_{i_r} \times \cdots \times P_{i_1} / B^r,\] where $B^r$ acts on $P_{i_r} \times \cdots \times P_{i_1}$ on the right by \[(p_r, \ldots, p_1)\cdot(b_r, \ldots, b_1) := (p_r b_r, b_r ^{-1} p_{r-1} b_{r-1}, \ldots, b_2 ^{-1} p_1 b_1)\] for $p_r \in P_{i_r}, \ldots, p_1 \in P_{i_1}$, and $b_r, \ldots, b_1 \in B$. It is well-known that the product map 
\begin{align}
\begin{aligned}
Z_{\bf i} \rightarrow G/B,\ (p_r, \ldots, p_1) \bmod B^r \mapsto p_r \cdots p_1 \bmod B,
\end{aligned}
\end{align} 
induces a birational morphism onto the Schubert variety $X(w) \subset G/B$; hence a valuation on $\c(X(w))$ can be identified with a valuation on $\c(Z_{\bf i})$. For the identity element $e \in G$, regard $U_{i_r} ^- \times \cdots \times U_{i_1} ^-$ as an affine open neighborhood of $(e, \ldots, e) \bmod B^r$ in $Z_{\bf i}$ by:
\begin{align}
\begin{aligned}
U_{i_r} ^- \times \cdots \times U_{i_1} ^- \hookrightarrow Z_{\bf i},\ (u_r, \ldots, u_1) \mapsto (u_r, \ldots, u_1) \bmod B^r.
\end{aligned}
\end{align}
By using the isomorphism of varieties $\c^r \xrightarrow{\sim} U_{i_r} ^- \times \cdots \times U_{i_1} ^-$, $(t_r, \ldots, t_1) \mapsto (\exp(t_r F_{i_r}), \ldots, \exp(t_1 F_{i_1}))$, we identify the function field $\c(Z_{\bf i}) = \c(U_{i_r} ^- \times \cdots \times U_{i_1} ^-)$ with the rational function field $\c(t_r, \ldots, t_1)$. Now we define a valuation $v_{\bf i}$ on $\c(X(w))$ to be the highest term valuation on $\c(t_r, \ldots, t_1)$ with respect to the lexicographic order $t_1 > \cdots > t_r$ (see Example \ref{highest term valuation}). In the following of this section, we describe $v_{\bf i}$ in terms of the Chevalley generators. If we set $w_{\ge k} := s_{i_r} s_{i_{r-1}} \cdots s_{i_k}$ for $1 \le k \le r$, then we obtain a sequence of subvarieties \[X(w_{\ge r}) \subset X(w_{\ge r-1}) \subset \cdots \subset X(w_{\ge 1}) = X(w).\] Note that the product map (1) induces a surjective birational morphism $Z_{(i_r, \ldots, i_k)} \twoheadrightarrow X(w_{\ge k})$, and that the open immersion (2) induces an open immersion $U_{i_r} ^- \times \cdots \times U_{i_k} ^- \hookrightarrow Z_{(i_r, \ldots, i_k)}$, where $Z_{(i_r, \ldots, i_k)}$ denotes the Bott-Samelson variety corresponding to the reduced word $(i_r, \ldots, i_k)$; hence the function field $\c(X(w_{\ge k}))$ is identified with the rational function field $\c(t_r, \ldots, t_k)$. Consider the right action of $U_{i_k} ^-$ on $U_{i_r} ^- \times \cdots \times U_{i_k} ^-$ given by $(u_r, \ldots, u_k) \cdot u = (u_r, \ldots, u_k u)$ for $u, u_k \in U_{i_k} ^-, u_{k+1} \in U_{i_{k+1}} ^-, \ldots, u_r \in U_{i_r} ^-$; this induces right actions of $U_{i_k} ^-$ and $\mathfrak{u}_{i_k} ^-$ on $\c[t_r, \ldots, t_k] = \c[U_{i_r} ^- \times \cdots \times U_{i_k} ^-]$, which are given by:
\begin{align}\label{derivation}
&f(t_r, \ldots, t_k) \cdot \exp(s F_{i_k}) = f(t_r, \ldots, t_{k+1}, t_k - s),\ {\rm and\ hence} \nonumber\\ 
&f(t_r, \ldots, t_k) \cdot F_{i_k} = - \frac{\partial}{\partial t_k} f(t_r, \ldots, t_k)
\end{align}
for $s \in \c$ and $f(t_r, \ldots, t_k) \in \c[t_r, \ldots, t_k]$. 

\vspace{2mm}\begin{prop}\label{valuation Chevalley}
For $f(t_r, \ldots, t_1) \in \c[t_r, \ldots, t_1]$, write $v_{\bf i} (f(t_r, \ldots, t_1)) = -(a_1, \ldots, a_r)$. Then, there hold the equalities
\begin{align*}
&a_1 = \max\{a \in \z_{\ge 0} \mid f(t_r, \ldots, t_1) \cdot F_{i_1} ^a \neq 0\},\\
&a_2 = \max\{a \in \z_{\ge 0} \mid (f(t_r, \ldots, t_1) \cdot F_{i_1} ^{a_1})|_{X(w_{\ge 2})} \cdot F_{i_2} ^a \neq 0\},\\
&\ \vdots\\
&a_r = \max\{a \in \z_{\ge 0} \mid (\cdots ((f(t_r, \ldots, t_1) \cdot F_{i_1} ^{a_1})|_{X(w_{\ge 2})} \cdot F_{i_2} ^{a_2}) \cdots)|_{X(w_{\ge r})} \cdot F_{i_r} ^a \neq 0\}.
\end{align*}
\end{prop}
\begin{proof}
It follows from the definition of $v_{\bf i}$ that $a_1$ is equal to the degree of $f(t_r, \ldots, t_1)$ with respect to the variable $t_1$. Therefore, we deduce that
\begin{align*}
a_1 &= \max\{a \in \z_{\ge 0} \mid \frac{\partial^a}{\partial t_1 ^a} f(t_r, \ldots, t_1) \neq 0\}\\
&= \max\{a \in \z_{\ge 0} \mid f(t_r, \ldots, t_1) \cdot F_{i_1} ^a \neq 0\}\quad ({\rm by\ equation}\ (\ref{derivation})).
\end{align*}
We remark that since the polynomial $f(t_r, \ldots, t_1) \cdot F_{i_1} ^{a_1}$ does not contain the variable $t_1$, the restriction $(f(t_r, \ldots, t_1) \cdot F_{i_1} ^{a_1})|_{X(w_{\ge 2})}$ is identical to $f(t_r, \ldots, t_1) \cdot F_{i_1} ^{a_1} \in \c[t_r, \ldots, t_2]$ as a polynomial in the variables $t_r, \ldots, t_2$; here the restriction map $\c(X(w)) \rightarrow \c(X(w_{\ge 2}))$ is given by $t_1 \mapsto 0$. Hence we see from the definition of $v_{\bf i}$ that $v_{(i_r, \ldots, i_2)}((f(t_r, \ldots, t_1) \cdot F_{i_1} ^{a_1})|_{X(w_{\ge 2})}) = -(a_2, \ldots, a_r)$, where $v_{(i_r, \ldots, i_2)}$ denotes the valuation on $\c(X(w_{\ge 2}))$ defined to be the highest term valuation on $\c(t_r, \ldots, t_2)$ $(= \c(U_{i_r} ^- \times \cdots \times U_{i_2} ^-))$ with respect to the lexicographic order $t_2 > \cdots > t_r$. By induction on $r = \ell(w)$, the assertion of the proposition follows.
\end{proof}

\section{Main result}
\subsection{Statement of the main result}

If we regard $\c$ as a $\q[q, q^{-1}]$-module by the natural $\q$-algebra homomorphism $\q[q, q^{-1}] \rightarrow \c$, $q \mapsto 1$, then the $\c$-vector space $V_q ^\q (\lambda) \otimes_{\q[q, q^{-1}]} \c$ has a $\mathfrak{g}$-module structure given by \[E_i(v \otimes c) := (e_i v) \otimes c,\ F_i(v \otimes c) := (f_i v) \otimes c,\ h_i(v \otimes c) := \left(\frac{t_i - t_i ^{-1}}{q_i - q_i ^{-1}}v\right) \otimes c\] for $i \in I$, $v \in V_q ^\q (\lambda)$, and $c \in \c$; this $\mathfrak{g}$-module is isomorphic to $V(\lambda)$ (see, for instance, \cite[Lemma 5.14]{J2}). We denote by $G^{\rm low} _\lambda (b) \in V(\lambda)$ the specialization of $G^{\rm low} _{q, \lambda} (b)$ at $q = 1$, i.e., $G^{\rm low} _\lambda (b) := G^{\rm low} _{q, \lambda} (b) \otimes 1 \in V_q ^\q (\lambda) \otimes_{\q[q, q^{-1}]} \c \simeq V(\lambda)$. Then, Proposition \ref{Demazure crystal} (4) implies that the set $\{G^{\rm low} _\lambda (b) \mid b \in \mathcal{B}_w (\lambda)\}$ forms a $\c$-basis of the Demazure module $V_w(\lambda)$. Also, we denote by $\{G^{\rm up} _{\lambda, w} (b) \mid b \in \mathcal{B}_w (\lambda)\} \subset V_w (\lambda) ^\ast = H^0(X(w), \mathcal{L}_\lambda)$ the dual basis, and by $\tau_\lambda \in V(\lambda) ^\ast = H^0(G/B, \mathcal{L}_\lambda)$ the lowest weight vector such that $\tau_\lambda(v_\lambda) = 1$; by restricting the section $\tau_\lambda$, we obtain a section in $H^0(X(w), \mathcal{L}_\lambda)$, which we denote by the same symbol $\tau_\lambda$. The following is the main result of this paper.

\vspace{2mm}\begin{thm}\label{main thm}
Let ${\bf i} \in I^r$ be a reduced word for $w \in W$, and $\lambda$ a dominant integral weight. Then, $\Psi_{\bf i} ^{(\lambda, w)} (b) = -v_{\bf i}(G^{\rm up} _{\lambda, w} (b)/\tau_\lambda)$ for all $b \in \mathcal{B}_w (\lambda)$.
\end{thm}\vspace{2mm}

We give a proof of this theorem in the next subsection.

\vspace{2mm}\begin{cor}\label{corollary1}
Let $\omega: \r \times \r^r \xrightarrow{\sim} \r \times \r^r$ be the linear automorphism given by $\omega(k, {\bf a}) = (k, -{\bf a})$. Then, $\mathcal{S}_{\bf i} ^{(\lambda, w)} = \omega(S(X(w), \mathcal{L}_\lambda, v_{\bf i}, \tau_\lambda))$, $\mathcal{C}_{\bf i} ^{(\lambda, w)} = \omega(C(X(w), \mathcal{L}_\lambda, v_{\bf i}, \tau_\lambda))$, and $\Delta_{\bf i} ^{(\lambda, w)} = -\Delta(X(w), \mathcal{L}_\lambda, v_{\bf i}, \tau_\lambda)$.
\end{cor}

\begin{proof}
Since $\mathcal{L}_{\lambda} ^{\otimes k} = \mathcal{L}_{k\lambda}$ and $\tau_\lambda ^k = \tau_{k\lambda}$ in $H^0(X(w), \mathcal{L}_{k\lambda})$ for all $k \in \z_{> 0}$, it follows that \[S(X(w), \mathcal{L}_\lambda, v_{\bf i}, \tau_\lambda) = \bigcup_{k>0} \{(k, v_{\bf i}(\sigma / \tau_{k\lambda})) \mid \sigma \in H^0(X(w), \mathcal{L}_{k\lambda}) \setminus \{0\}\}.\] Also, since $\Psi_{\bf i} ^{(k\lambda, w)} (b)$, $b \in \mathcal{B}_w (k\lambda)$, are all distinct, we deduce from Proposition \ref{prop1,val} (2) and Theorem \ref{main thm} that \[\{\Psi_{\bf i} ^{(k\lambda, w)} (b) \mid b \in \mathcal{B}_w (k\lambda)\} = \{-v_{\bf i} (\sigma/\tau_{k\lambda}) \mid \sigma \in H^0(X(w), \mathcal{L}_{k\lambda}) \setminus \{0\}\}\] for all $k \in \z_{> 0}$, which implies that $\mathcal{S}_{\bf i} ^{(\lambda, w)} = \omega(S(X(w), \mathcal{L}_\lambda, v_{\bf i}, \tau_\lambda))$. From this equality, the other assertions follow immediately by the definitions. 
\end{proof}

\begin{cor}\label{corollary}
Let ${\bf i} \in I^r$ be a reduced word for $w \in W$, and $\lambda$ a dominant integral weight.
\begin{enumerate}
\item[{\rm (1)}] The sets $\mathcal{S}_{\bf i} ^{(\lambda, w)}$ and $S(X(w), \mathcal{L}_\lambda, v_{\bf i}, \tau_\lambda)$ are both finitely generated semigroups.  
\item[{\rm (2)}] The real closed cones $\mathcal{C}_{\bf i} ^{(\lambda, w)}$ and $C(X(w), \mathcal{L}_\lambda, v_{\bf i}, \tau_\lambda)$ are both rational convex polyhedral cones, and the equality $S(X(w), \mathcal{L}_\lambda, v_{\bf i}, \tau_\lambda) = C(X(w), \mathcal{L}_\lambda, v_{\bf i}, \tau_\lambda) \cap (\z_{>0} \times \z^r)$ holds.
\item[{\rm (3)}] The compact sets $\Delta_{\bf i} ^{(\lambda, w)}$ and $\Delta(X(w), \mathcal{L}_\lambda, v_{\bf i}, \tau_\lambda)$ are both rational convex polytopes, and the equality $\Psi_{\bf i} ^{(\lambda, w)} (\mathcal{B}_w (\lambda)) = -\Delta(X(w), \mathcal{L}_\lambda, v_{\bf i}, \tau_\lambda) \cap \z^r$ holds.
\end{enumerate}
\end{cor}

\begin{proof}
Part (2) follows from Corollary \ref{a finite union} (2) and from the fact that $C(X(w), \mathcal{L}_\lambda, v_{\bf i}, \tau_\lambda)$ is convex. Then, part (3) is an immediate consequence of part (2) and Corollary \ref{a finite union of polytopes}. Finally, part (1) follows from part (2) and from Gordan's lemma (see, for instance, \cite[Proposition 1.2.17]{CLS}). 
\end{proof}

\subsection{Proof of Theorem \ref{main thm}}

We make use of the lower global basis of $U_q (\mathfrak{u}^-)$. Let $U^- \subset G$ be the unipotent radical of the opposite Borel subgroup, $\mathfrak{u}^- \subset \mathfrak{g}$ the Lie algebra of $U^-$, and $U(\mathfrak{u}^-)$ the universal enveloping algebra of $\mathfrak{u}^-$; this $\c$-algebra is isomorphic to $U_q ^\q(\mathfrak{u}^-) \otimes_{\q[q, q^{-1}]} \c$ by identifying $F_i$ with $f_i \otimes 1$. For $b \in \mathcal{B}(\infty)$, denote by $G^{\rm low} (b) \in U(\mathfrak{u}^-)$ the specialization of $G_q ^{\rm low} (b)$ at $q = 1$, i.e., $G^{\rm low} (b) := G_q ^{\rm low} (b) \otimes 1 \in U_q ^\q(\mathfrak{u}^-) \otimes_{\q[q, q^{-1}]} \c \simeq U(\mathfrak{u}^-)$. The algebra $U(\mathfrak{u}^-)$ has a Hopf algebra structure given by the following coproduct $\Delta$, counit $\varepsilon$, and antipode $S$: 
\[\Delta(F_i) = F_i \otimes 1 + 1 \otimes F_i,\ \varepsilon(F_i) = 0,\ {\rm and}\ S(F_i) = -F_i\]
for $i \in I$. Also, we can regard $U(\mathfrak{u}^-)$ as a multigraded $\c$-algebra: \[U(\mathfrak{u}^-) = \bigoplus_{{\bf d} \in \z^I _{\ge 0}} U(\mathfrak{u}^-)_{\bf d},\] where the homogeneous component $U(\mathfrak{u}^-)_{\bf d}$ for ${\bf d} = (d_i)_{i \in I} \in \z^I _{\ge 0}$ is defined to be the $\c$-subspace of $U(\mathfrak{u}^-)$ spanned by elements $F_{j_1} \cdots F_{j_{|{\bf d}|}}$ for which the cardinality of $\{1 \le l \le |{\bf d}| \mid j_l = i\}$ is equal to $d_i$ for all $i \in I$; here we set $|{\bf d}| := \sum_{i \in I} d_i$. Let \[U(\mathfrak{u}^-)^\ast _{\rm gr} := \bigoplus_{{\bf d} \in \z^I _{\ge 0}} U(\mathfrak{u}^-)_{\bf d} ^\ast\] be the graded dual of $U(\mathfrak{u}^-)$ endowed with the dual Hopf algebra structure, and $\{G^{\rm up} (b) \mid b \in \mathcal{B}(\infty)\} \subset U(\mathfrak{u}^-)^\ast _{\rm gr}$ the dual basis of $\{G^{\rm low} (b) \mid b \in \mathcal{B}(\infty)\} \subset U(\mathfrak{u}^-)$. Proposition \ref{properties of global bases} (2) implies that $\pi_\lambda(G^{\rm low} (b)) = G^{\rm low} _\lambda (\pi_\lambda(b))$ for all $b \in \widetilde{\mathcal{B}} (\lambda)$, and that $\pi_\lambda(G^{\rm low} (b)) = 0$ for all $b \in \mathcal{B}(\infty) \setminus \widetilde{\mathcal{B}} (\lambda)$, where by abuse of notation, we denote by $\pi_\lambda: U(\mathfrak{u}^-) \twoheadrightarrow V(\lambda)$ the surjective $U(\mathfrak{u}^-)$-module homomorphism given by $\pi_\lambda(u) = u \cdot v_\lambda$ for $u \in U(\mathfrak{u}^-)$. Hence we obtain $\pi_\lambda ^\ast (G^{\rm up} _\lambda (\pi_\lambda(b))) = G^{\rm up} (b)$ for all $b \in \widetilde{\mathcal{B}}(\lambda)$, where $\pi_\lambda ^\ast: V(\lambda) ^\ast \hookrightarrow U(\mathfrak{u}^-)^\ast _{\rm gr}$ denotes the dual of $\pi_\lambda$. Note that the coordinate ring $\c[U^-]$ has a Hopf algebra structure given by the following coproduct $\Delta$, counit $\varepsilon$, and antipode $S$: 
\[\Delta(f) (u_1 \otimes u_2) = f(u_1 u_2),\ \varepsilon(f) = f(1),\ {\rm and}\ S(f) (u) = f(u^{-1})\]
for $f \in \c[U^-]$ and $u, u_1, u_2 \in U^-$. It is well-known that this Hopf algebra is isomorphic to the dual Hopf algebra $U(\mathfrak{u}^-)_{\rm gr} ^\ast$ as follows.

\vspace{2mm}\begin{lem}[{see, for instance, \cite[Proposition 5.1]{GLS}}]\label{lemma1}
Define a map $\Upsilon: U(\mathfrak{u}^-)_{\rm gr} ^\ast \rightarrow \c[U^-]$ by $\Upsilon(\rho)(\exp(x)) = \sum_{l \ge 0} \rho(x^l)/l!$ for $\rho \in U(\mathfrak{u}^-)_{\rm gr} ^\ast$ and $x \in \mathfrak{u}^-;$ here, $\exp(x) \in U^-$ and $x^l \in U(\mathfrak{u}^-)$ for all $l \in \z_{\ge 0}$. Then, the map $\Upsilon$ is an isomorphism of Hopf algebras.
\end{lem}\vspace{2mm}

Using the isomorphism $\Upsilon$, we identify $U(\mathfrak{u}^-)_{\rm gr} ^\ast$ with $\c[U^-]$. Then, the space $H^0(G/B, \mathcal{L}_\lambda)$ of global sections can be regarded as a $\c$-subspace of $\c[U^-]$ by: $H^0(G/B, \mathcal{L}_\lambda) = V(\lambda)^\ast \lhook\joinrel\xrightarrow{\pi_\lambda ^\ast} U(\mathfrak{u}^-)_{\rm gr} ^\ast = \c[U^-]$. Also, we can regard $U^-$ as an affine open subset of $G/B$ by: \[U^- \hookrightarrow G/B,\ u \mapsto u \bmod B;\] hence we obtain an isomorphism \[\c(G/B) \xrightarrow{\sim} \c(U^-),\ f \mapsto f|_{U^-}.\] 

\vspace{2mm}\begin{lem}\label{lemma2}
The equality $(\tau/\tau_\lambda)|_{U^-} = (\Upsilon \circ \pi_\lambda ^\ast)(\tau)$ holds in $\c[U^-]$ for all $\tau \in H^0(G/B, \mathcal{L}_\lambda)$. In particular, the restriction of $G^{\rm up} _{\lambda} (\pi_\lambda(b))/\tau_\lambda$ to $U^-$ is identical to $\Upsilon(G^{\rm up} (b)) \in \c[U^-]$ for all $b \in \widetilde{\mathcal{B}}(\lambda)$.
\end{lem}

\begin{proof}
Because $\exp(x) \cdot v_\lambda \in v_\lambda + x \cdot V(\lambda)$ and $\tau_\lambda (x \cdot V(\lambda)) = \{0\}$ for $x \in \mathfrak{u}^-$, we have $\tau_\lambda(\exp(x) \cdot v_\lambda) = 1$. Therefore, as an element of $H^0(G/B, \mathcal{L}_\lambda)$, the section $\tau_\lambda$ does not vanish on $U^-$ $(\hookrightarrow G/B)$; in particular, $(\tau/\tau_\lambda)|_{U^-} \in \c[U^-]$. Furthermore, we see that
\begin{align*}
(\tau/\tau_\lambda)(\exp(x)) &= \tau(\exp(x) \cdot v_\lambda)/\tau_\lambda(\exp(x) \cdot v_\lambda)\\
&({\rm by\ the\ definition\ of\ the\ isomorphism}\ \rho_\lambda ^\ast: V(\lambda)^\ast \xrightarrow{\sim} H^0(G/B, \mathcal{L}_\lambda)\ {\rm in}\ \S\S 3.1)\\
&= \tau(\exp(x) \cdot v_\lambda)\quad({\rm since}\ \tau_\lambda(\exp(x) \cdot v_\lambda) = 1).
\end{align*}
Also, we have
\begin{align*}
(\Upsilon \circ \pi_\lambda ^\ast) (\tau) (\exp(x)) &= \sum_{l \ge 0} (\pi_\lambda ^\ast(\tau))(x^l)/l!\quad ({\rm by\ the\ definition\ of}\ \Upsilon)\\
&= \sum_{l \ge 0} \tau(x^l \cdot v_\lambda)/l!\quad ({\rm since}\ \pi_\lambda(x^l) = x^l \cdot v_\lambda)\\
&= \tau(\exp(x) \cdot v_\lambda).
\end{align*}
From these, the assertion of the lemma follows immediately.
\end{proof}
Let $\langle \cdot, \cdot \rangle: U(\mathfrak{u}^-)^\ast _{\rm gr} \times U(\mathfrak{u}^-) \rightarrow \c$ denote the canonical pairing. Define a $U(\mathfrak{u}^-)$-bimodule structure on $U(\mathfrak{u}^-)^\ast _{\rm gr}$ by 
\begin{align*}
&\langle x \cdot \rho, y \rangle := - \langle \rho, x \cdot y \rangle,\ {\rm and}\\
&\langle \rho \cdot x, y \rangle := - \langle \rho, y \cdot x \rangle
\end{align*}
for $x \in \mathfrak{u}^-$, $\rho \in U(\mathfrak{u}^-)^\ast _{\rm gr}$, and $y \in U(\mathfrak{u}^-)$. Also, the coordinate ring $\c[U^-]$ has a natural $U^-$-bimodule structure, which is given by 
\begin{align*}
&(u_1 \cdot f)(u_2) := f(u_1 ^{-1} u_2),\ {\rm and}\\
&(f \cdot u_1)(u_2) := f(u_2 u_1 ^{-1})
\end{align*}
for $u_1, u_2 \in U^-$ and $f \in \c[U^-]$. This induces a $U(\mathfrak{u}^-)$-bimodule structure on $\c[U^-]$. Note that $\Upsilon$ is an isomorphism of $U(\mathfrak{u}^-)$-bimodules.

\vspace{2mm}\begin{lem}\label{lemma3}
For $k \ge 0$, $i \in I$, and $b \in \mathcal{B}(\infty)$, 
\begin{align*}
&F_i ^{(k)} \cdot G^{\rm up} (b) \in (-1)^{k} \genfrac{(}{)}{0pt}{}{\varepsilon_i (b)}{k} G^{\rm up} (\tilde{e}_i ^k b) + \sum_{\substack{b^\prime \in \mathcal{B}(\infty);\ {\rm wt}(b^\prime) = {\rm wt}(\tilde{e}_i ^k b),\\ \varepsilon_i (b^\prime) < \varepsilon_i (\tilde{e}_i ^k b)}} \z G^{\rm up} (b^\prime),\ {\it and}\\
&G^{\rm up} (b) \cdot F_i ^{(k)} \in (-1)^{k} \genfrac{(}{)}{0pt}{}{\varepsilon_i ^\ast (b)}{k} G^{\rm up} ((\tilde{e}_i ^\ast)^k b) + \sum_{\substack{b^\prime \in \mathcal{B}(\infty);\ {\rm wt}(b^\prime) = {\rm wt}((\tilde{e}_i ^\ast)^k b),\\ \varepsilon_i ^\ast (b^\prime) < \varepsilon_i ^\ast ((\tilde{e}_i ^\ast)^k b)}} \z G^{\rm up} (b^\prime). 
\end{align*}
Here, $F_i ^{(k)} := F_i ^k /k!$, and $\genfrac{(}{)}{0pt}{}{\varepsilon_i (b)}{k}$, $\genfrac{(}{)}{0pt}{}{\varepsilon_i ^\ast (b)}{k}$ are usual binomial coefficients. In particular, there hold the equalities
\begin{align*}
&\varepsilon_i (b) = \max \{k \in \z_{\ge 0} \mid F_i ^{(k)} \cdot G^{\rm up} (b) \neq 0\},\ F_i ^{(\varepsilon_i (b))} \cdot G^{\rm up} (b) = (-1)^{\varepsilon_i (b)} G^{\rm up} (\tilde{e}_i ^{\varepsilon_i (b)} b),\ {\it and}\\
&\varepsilon_i ^\ast (b) = \max \{k \in \z_{\ge 0} \mid G^{\rm up} (b) \cdot F_i ^{(k)} \neq 0\},\ G^{\rm up} (b) \cdot F_i ^{(\varepsilon_i ^\ast (b))} = (-1)^{\varepsilon_i ^\ast (b)} G^{\rm up} ((\tilde{e}_i ^\ast)^{\varepsilon_i ^\ast (b)} b).
\end{align*}
\end{lem}

\begin{proof}
We prove only the assertion for $G^{\rm up} (b) \cdot F_i ^{(k)}$; the proof of the assertion for $F_i ^{(k)} \cdot G^{\rm up} (b)$ is similar. For $b^\prime \in \mathcal{B}(\infty)$, the coefficient of $G^{\rm up} (b^\prime)$ in $G^{\rm up} (b) \cdot F_i ^{(k)}$ is equal to \[\langle G^{\rm up} (b) \cdot F_i ^{(k)}, G^{\rm low}(b^\prime)\rangle = (-1)^{k} \langle G^{\rm up} (b), G^{\rm low} (b^\prime) \cdot F_i ^{(k)}\rangle.\] If $\langle G^{\rm up} (b), G^{\rm low} (b^\prime) \cdot F_i ^{(k)}\rangle \neq 0$, then we see that ${\rm wt}(b) = {\rm wt}((\tilde{f}_i ^\ast)^k b^\prime)$, i.e., ${\rm wt}(b^\prime) = {\rm wt}((\tilde{e}_i ^\ast)^k b)$ by the definition of the graded dual $U(\mathfrak{u}^-)_{\rm gr} ^\ast$. Moreover, in this case, Proposition \ref{properties of global bases} (4) implies that $b = (\tilde{f}_i ^\ast)^k b^\prime$ or $\varepsilon_i ^\ast (b) > \varepsilon_i ^\ast ((\tilde{f}_i ^\ast)^k b^\prime)$, i.e., $b^\prime = (\tilde{e}_i ^\ast)^k b$ or $\varepsilon_i ^\ast (b^\prime) < \varepsilon_i ^\ast ((\tilde{e}_i ^\ast)^k b)$. Therefore, we deduce that \[G^{\rm up} (b) \cdot F_i ^{(k)} \in \c G^{\rm up} ((\tilde{e}_i ^\ast)^k b) + \sum_{\substack{b^\prime \in \mathcal{B}(\infty);\ {\rm wt}(b^\prime) = {\rm wt}((\tilde{e}_i ^\ast)^k b),\\ \varepsilon_i ^\ast (b^\prime) < \varepsilon_i ^\ast ((\tilde{e}_i ^\ast)^k b)}} \c G^{\rm up} (b^\prime).\] The claim for the coefficients also follows from the same proposition.
\end{proof}

Note that in Lemma \ref{lemma3}, $F_i ^{(k)} \cdot G^{\rm up} (b) \neq 0$ (resp., $G^{\rm up} (b) \cdot F_i ^{(k)} \neq 0$) if and only if $F_i ^{k} \cdot G^{\rm up} (b) \neq 0$ (resp., $G^{\rm up} (b) \cdot F_i ^{k} \neq 0$) since $F_i ^{(k)} = F_i ^k /k!$. 

If $w_0 \in W$ is the longest element, then the Schubert variety $X(w_0)$ is just the full flag variety $G/B$. For a reduced word ${\bf i} = (i_N, \ldots, i_1)$ for $w_0$, we deduce from the argument in \S\S 3.2 that the product map \[U_{i_N} ^- \times \cdots \times U_{i_1} ^- \rightarrow U^-,\ (u_N, \ldots, u_1) \mapsto u_N \cdots u_1,\] is birational, and hence that the coordinate ring $\c[U^-]$ can be regarded as a $\c$-subalgebra of $\c[U_{i_N} ^- \times \cdots \times U_{i_1} ^-] = \c[t_N, \ldots, t_1]$. In this way, we think of $G^{\rm up}(b) \in U(\mathfrak{u}^-)_{\rm gr} ^\ast = \c[U^-]$ for $b \in \mathcal{B}(\infty)$ as a polynomial in the variables $t_N, \ldots, t_1$.

\vspace{2mm}\begin{prop}\label{Verma case}
For a reduced word ${\bf i} = (i_N, \ldots, i_1)$ for $w_0$ and $b \in \mathcal{B}(\infty)$, the Kashiwara embedding $\Psi_{\bf i}(b)$ is equal to $-v_{\bf i}(G^{\rm up}(b))$.
\end{prop}

\begin{proof}
First, observe that the right action of $F_{i_1}$ on $\c[U^-]$ $(\hookrightarrow \c[U_{i_N} ^- \times \cdots \times U_{i_1} ^-])$, which is induced by the right $U(\mathfrak{u}^-)$-module structure on $\c[U^-]$, is identical to that of $F_{i_1}$ on $\c[U_{i_N} ^- \times \cdots \times U_{i_1} ^-]$, which is defined in \S\S 3.2. Therefore, if we write $\Psi_{\bf i} (b) = (a_1, \ldots, a_N)$ and $v_{\bf i} (G^{\rm up} (b)) = -(a_1 ^\prime, \ldots, a_N ^\prime)$ for $b \in \mathcal{B}(\infty)$, then we deduce that
\begin{align*}
a_1 ^\prime &= \max\{a \in \z_{\ge 0} \mid G^{\rm up} (b) \cdot F_{i_1} ^{a} \neq 0\}\quad ({\rm by\ Proposition}\ \ref{valuation Chevalley})\\
&= \max\{a \in \z_{\ge 0} \mid (\tilde{e}_{i_1} ^\ast)^a b \neq 0\}\quad ({\rm by\ the\ assertion\ of\ Lemma}\ \ref{lemma3}\ {\rm for}\ \varepsilon_i ^\ast(b))\\ 
&= a_1\quad ({\rm by\ Proposition}\ \ref{star string}\ (2)).
\end{align*} 
Repeating this argument, with $b$ and $U^-$ replaced by $(\tilde{e}_{i_1} ^\ast)^{a_1} b$ and $\overline{U_{i_N} ^- \cdots U_{i_2} ^-}$ $(\subset U^-)$, respectively, we obtain $(a_1 ^\prime, \ldots, a_N ^\prime) = (a_1, \ldots, a_N)$; here we use the equality \[G^{\rm up} (b) \cdot F_{i_1} ^{(a_1)} = (-1)^{a_1} G^{\rm up}((\tilde{e}_{i_1} ^\ast)^{a_1} b)\quad({\rm by\ the\ assertion\ of\ Lemma}\ \ref{lemma3}\ {\rm for}\ G^{\rm up} (b) \cdot F_i ^{(\varepsilon_i ^\ast (b))})\] and the fact that the restriction of $G^{\rm up} (b) \cdot F_{i_1} ^{(a_1)}$ to $\overline{U_{i_N} ^- \cdots U_{i_2} ^-}$ is identical to $G^{\rm up} (b) \cdot F_{i_1} ^{(a_1)} \in \c[t_N, \ldots, t_2]$ as a polynomial in the variables $t_N, \ldots, t_2$ (see also the proof of Proposition \ref{valuation Chevalley}). This proves the proposition. 
\end{proof}

Since $\Psi_{\bf i}(b)$, $b \in \mathcal{B}(\infty)$, are all distinct, we obtain the following from Proposition \ref{prop1,val} (2).

\vspace{2mm}\begin{cor}\label{cor1}
The polyhedral realization $\Psi_{\bf i}(\mathcal{B}(\infty))$ is identical to $-v_{\bf i}(\c[U^-] \setminus \{0\})$.
\end{cor}\vspace{2mm}

Recall that the map $\Psi_{\bf i} ^{(\lambda, w_0)}$ (see Definition \ref{definition2}) gives an embedding of $\mathcal{B}(\lambda)$ into $\z^N$, which we denote simply by $\Psi_{\bf i} ^{(\lambda)}$. It follows from the definition of $\Psi_{\bf i} ^{(\lambda, w_0)}$ that $\Psi_{\bf i} ^{(\lambda)} (\pi_\lambda(b)) = \Psi_{\bf i} (b)$ for all $b \in \widetilde{\mathcal{B}}(\lambda)$. Hence we obtain the following by Lemma \ref{lemma2}.

\vspace{2mm}\begin{cor}\label{the case of longest}
The Kashiwara embedding $\Psi_{\bf i} ^{(\lambda)}(b)$ is equal to $-v_{\bf i}(G^{\rm up} _{\lambda}(b)/\tau_\lambda)$ for all $b \in \mathcal{B}(\lambda)$.
\end{cor}\vspace{2mm}

We remark that this corollary is precisely the assertion of Theorem \ref{main thm} in the case $w = w_0$. In the rest of this subsection, we give a proof of Theorem \ref{main thm} for an arbitrary $w \in W$. Let ${\bf i} = (i_r, \ldots, i_1)$ be a reduced word for $w$, and extend it to a reduced word $\hat{\bf i} = (i_N, \ldots, i_{r+1}, i_r, \ldots, i_1)$ for $w_0$. By the same argument as the one preceding Proposition \ref{Verma case}, the reduced word $\hat{\bf i}$ gives an injective $\c$-algebra homomorphism $\c[U^-] \hookrightarrow \c[t_N, \ldots, t_1]$. Moreover, the coordinate ring $\c[X(w) \cap U^-]$ can be identified with a $\c$-subalgebra of $\c[t_r, \ldots, t_1]$ $(\subset \c[t_N, \ldots, t_1])$ by using the birational morphism \[U_{i_r} ^- \times \cdots \times U_{i_1} ^- \rightarrow X(w) \cap U^-,\ (u_r, \ldots, u_1) \mapsto u_r \cdots u_1 \bmod B;\] note that the restriction map $\c[U^-] \twoheadrightarrow \c[X(w) \cap U^-]$ is given by 
\begin{align*}
t_k \mapsto 
\begin{cases}
t_k &{\rm if}\ 1 \le k \le r,\\
0 &{\rm if}\ r+1 \le k \le N.
\end{cases}
\end{align*}
Also, recall that the highest term valuation $v_{\hat{\bf i}}$ (resp., $v_{\bf i}$) is a valuation on $\c(t_N, \ldots, t_1)$ (resp., $\c(t_r, \ldots, t_1)$). For $b \in \mathcal{B}_w(\lambda)$, we have 
\begin{align*}
-v_{\hat{\bf i}}(G^{\rm up} _{\lambda}(b)/\tau_\lambda) &= \Psi_{\hat{\bf i}} ^{(\lambda)}(b)\quad ({\rm by\ Corollary}\ \ref{the case of longest})\\
&= (\Psi_{\bf i} ^{(\lambda, w)}(b), \underbrace{0, \ldots, 0}_{N-r})\quad ({\rm by\ Definition}\ \ref{definition2}).
\end{align*}
Therefore, the highest term of $G^{\rm up} _{\lambda}(b)/\tau_\lambda \in \c[t_N, \ldots, t_1]$ with respect to the lexicographic order $t_1 > \cdots > t_N$ does not contain the variables $t_k$ for $r+1 \le k \le N$, which implies that \[-v_{\bf i}((G^{\rm up} _{\lambda}(b)/\tau_\lambda)|_{X(w) \cap U^-}) = \Psi_{\bf i} ^{(\lambda, w)}(b).\] Now we remark that the restriction map $V(\lambda)^\ast = H^0(G/B, \mathcal{L}_\lambda) \rightarrow H^0(X(w), \mathcal{L}_\lambda) = V_w(\lambda)^\ast$ is identical to the dual of the inclusion map $V_w(\lambda) \hookrightarrow V(\lambda)$. This implies that $G^{\rm up} _{\lambda}(b)|_{X(w)} = G^{\rm up} _{\lambda, w}(b)$, and hence that $(G^{\rm up} _{\lambda}(b)/\tau_\lambda)|_{X(w) \cap U^-} = G^{\rm up} _{\lambda, w}(b)/\tau_\lambda$; here recall that the section $\tau_\lambda \in H^0(X(w), \mathcal{L}_\lambda)$ is defined to be the restriction of the section $\tau_\lambda \in H^0(G/B, \mathcal{L}_\lambda)$. Thus, we obtain $-v_{\bf i}(G^{\rm up} _{\lambda, w}(b)/\tau_\lambda) = \Psi_{\bf i} ^{(\lambda, w)}(b)$. This completes the proof of Theorem \ref{main thm}.

\section{Applications}
\subsection{Explicit forms of Newton-Okounkov convex bodies}

Under the assumption that $(\tilde{\bf i}, \lambda)$ is ample (see Definition \ref{definition of ample} below), the polyhedral realization $\Psi_{\tilde{\bf i}} ^{(\lambda)} (\mathcal{B}_w (\lambda))$ in \S\S 2.3 is given by a system of explicit affine inequalities. In order to obtain explicit forms of Newton-Okounkov convex bodies, we recall the description of $\Psi_{\tilde{\bf i}} ^{(\lambda)} (\mathcal{B}_w (\lambda))$, following \cite{N1} and \cite{N2}. Consider the infinite-dimensional vector space \[\q^{\infty} := \{{\bf a} = (\ldots, a_k, \ldots, a_2, a_1) \mid a_k \in \q\ {\rm and}\ a_k = 0\ {\rm for}\ k \gg 0\},\] and write an affine function $\psi$ on $\q^{\infty}$ as $\psi({\bf a}) = \psi_0 + \sum_{k \ge 1} \psi_k a_k$, with $\psi_0, \psi_1, \ldots \in \q$. Recall that ${\bf i} = (i_r, \ldots, i_1)$ is a reduced word for $w \in W$, and that $\tilde{\bf i} = (\ldots, i_k, \ldots, i_{r+1}, i_r, \ldots, i_1)$ is an extension of ${\bf i}$ such that $i_k \neq i_{k+1}$ for all $k \ge 1$, and such that the cardinality of $\{k \ge 1 \mid i_k = i\}$ is $\infty$ for each $i \in I$. For $k \ge 1$, we set 
\begin{align*}
k^{(+)} &:= \min\{l > k \mid i_l = i_k\},\ {\rm and}\\
k^{(-)} &:= 
\begin{cases}
\max\{l < k \mid i_l = i_k\} &{\rm if\ it\ exists},\\
0 &{\rm otherwise}.
\end{cases}
\end{align*}
For $k \ge 1$ and $i \in I$, let $\beta_k ^{(\pm)} ({\bf a})$, $\lambda^{(i)} ({\bf a})$ denote the affine functions given by 
\begin{align*}
&\beta_k ^{(+)} ({\bf a}) := a_k + \sum_{k < j < k^{(+)}} \langle \alpha_{i_j}, h_{i_k} \rangle a_j + a_{k^{(+)}},\\
&\beta_k ^{(-)} ({\bf a}) :=
\begin{cases}
a_{k^{(-)}} + \sum_{k^{(-)} < j < k} \langle \alpha_{i_j}, h_{i_k} \rangle a_j + a_k &{\rm if}\ k^{(-)} > 0,\\
-\langle\lambda, h_{i_k} \rangle + \sum_{1 \le j < k} \langle \alpha_{i_j}, h_{i_k} \rangle a_j + a_k &{\rm if}\ k^{(-)} = 0,
\end{cases}\\
&\lambda^{(i)} ({\bf a}) := \langle\lambda, h_i\rangle - \sum_{1 \le j < \tilde{{\bf i}}^{(i)}} \langle\alpha_{i_j}, h_i\rangle a_j - a_{\tilde{{\bf i}}^{(i)}},
\end{align*}
where we write $\tilde{\bf i}^{(i)} := \min\{k \ge 1 \mid i_k = i\}$. Define operators $\widehat{S}_k$, $k \ge 1$, for affine functions on $\q^{\infty}$ by
\begin{align*}
\widehat{S}_k(\psi) := 
\begin{cases}
\psi - \psi_k \beta_k ^{(+)} &{\rm if}\ \psi_k >0,\\
\psi - \psi_k \beta_k ^{(-)} &{\rm if}\ \psi_k \le 0,
\end{cases}
\end{align*}
and let $\Xi_{\tilde{\bf i}} [\lambda]$ denote the set of all affine functions generated by $\widehat{S}_k$, $k \ge 1$, from the functions $a_j$, $j \ge 1$, and $\lambda^{(i)}({\bf a})$, $i \in I$; namely, 
\begin{align*}
\Xi_{\tilde{\bf i}} [\lambda] &:= \{\widehat{S}_{j_k} \cdots \widehat{S}_{j_1} a_{j_0} \mid k \ge 0\ {\rm and}\ j_0, \ldots, j_k \ge 1\}\\
&\cup \{\widehat{S}_{j_k} \cdots \widehat{S}_{j_1} \lambda^{(i)} ({\bf a}) \mid k \ge 0,\ i \in I,\ {\rm and}\ j_1, \ldots, j_k \ge 1\}.
\end{align*}

\vspace{2mm}\begin{defi}\normalfont\label{definition of ample}
Set \[\Sigma_{\tilde{\bf i}}[\lambda] := \{{\bf a} \in \z_{\tilde{\bf i}} ^\infty [\lambda] \subset \q^\infty \mid \psi ({\bf a}) \ge 0\ {\rm for\ all}\ \psi \in \Xi_{\tilde{\bf i}} [\lambda]\};\]
a pair $(\tilde{\bf i}, \lambda)$ is called {\it ample} if $(\ldots, 0, \ldots, 0, 0) \in \Sigma_{\tilde{\bf i}}[\lambda]$.
\end{defi}\vspace{2mm}

\begin{prop}[{see \cite[Theorem 4.1]{N1} and \cite[Proposition 3.1]{N2}}]\label{ample}
Assume that $(\tilde{\bf i}, \lambda)$ is ample. Then, the image $\Psi_{\tilde{\bf i}} ^{(\lambda)} (\mathcal{B}_w(\lambda))$ is identical to the set \[\{(\ldots, a_k, \ldots, a_2, a_1) \in \Sigma_{\tilde{\bf i}}[\lambda] \mid a_k = 0\ {\it for\ all}\ k > r\}.\]
\end{prop}\vspace{2mm}

For all $\psi \in \Xi_{\tilde{\bf i}} [\lambda]$, the constant term $\psi(\ldots, 0, \ldots, 0, 0)$ is regarded as a linear function of $\lambda$ by the definition of $\Xi_{\tilde{\bf i}} [\lambda]$; hence, for a fixed dominant integral weight $\lambda$, we can regard an element of $\Xi_{\tilde{\bf i}} [k\lambda]$ as a linear function of $k$ and $a_j$, $j \ge 1$. Therefore, we obtain the following from Definitions \ref{definition2}, \ref{definition3} and Corollary \ref{corollary1}.

\vspace{2mm}\begin{cor}\label{explicit description}
Assume that $(\tilde{\bf i}, \lambda)$ is ample. Then, 
\begin{align*}
\mathcal{S}_{\bf i} ^{(\lambda, w)} &= \omega(S(X(w), \mathcal{L}_\lambda, v_{\bf i}, \tau_\lambda))\\
&= \{(k, a_1, \ldots, a_r) \in \z_{>0} \times \z^r \mid \psi(\ldots, 0, 0, a_r, \ldots, a_1) \ge 0\ {\it for\ all}\ \psi \in \Xi_{\tilde{\bf i}} [k\lambda]\},\\
\mathcal{C}_{\bf i} ^{(\lambda, w)} &= \omega(C(X(w), \mathcal{L}_\lambda, v_{\bf i}, \tau_\lambda))\\
&= \{(k, a_1, \ldots, a_r) \in \r_{\ge 0} \times \r^r \mid \psi(\ldots, 0, 0, a_r, \ldots, a_1) \ge 0\ {\it for\ all}\ \psi \in \Xi_{\tilde{\bf i}} [k\lambda]\},\ {\it and}\\
\Delta_{\bf i} ^{(\lambda, w)} &= -\Delta(X(w), \mathcal{L}_\lambda, v_{\bf i}, \tau_\lambda)\\
&= \{(a_1, \ldots, a_r) \in \r^r \mid \psi(\ldots, 0, 0, a_r, \ldots, a_1) \ge 0\ {\it for\ all}\ \psi \in \Xi_{\tilde{\bf i}} [\lambda]\}.
\end{align*}
\end{cor}

\subsection{Relation with Kashiwara's involution}

This subsection is devoted to describing Kashiwara's involution $\ast$ in terms of valuations on the rational function field. We first recall the main result of \cite{Kav}. Let ${\bf i} = (i_r, \ldots, i_1)$ be a reduced word for $w \in W$, and $\Phi_{\bf i} :\mathcal{B}_w (\infty) \hookrightarrow \z^r$ the corresponding string parameterization (see \S\S 2.3). We set \[\widetilde{\mathcal{S}}_{\bf i} ^{(\lambda, w)} := \bigcup_{k>0} \{(k, \Phi_{\bf i} (b)) \mid b \in \widetilde{\mathcal{B}}_w (k\lambda)\} \subset \z_{>0} \times \z^r,\] and denote by $\widetilde{\mathcal{C}}_{\bf i} ^{(\lambda, w)} \subset \r_{\ge 0} \times \r^r$ the smallest real closed cone containing $\widetilde{\mathcal{S}}_{\bf i} ^{(\lambda, w)}$; here we write $\widetilde{\mathcal{B}}_w (k\lambda) := \mathcal{B}_w (\infty) \cap \widetilde{\mathcal{B}}(k\lambda)$.

\vspace{2mm}\begin{defi}[{see [Lit, Section 1] and [Kav, Definition 3.5]}]\normalfont
Let ${\bf i} \in I^r$ be a reduced word for $w \in W$, and $\lambda$ a dominant integral weight. Define a subset $\widetilde{\Delta}_{\bf i} ^{(\lambda, w)} \subset \r^r$ by \[\widetilde{\Delta}_{\bf i} ^{(\lambda, w)} := \{{\bf a} \in \r^r \mid (1, {\bf a}) \in \widetilde{\mathcal{C}}_{\bf i} ^{(\lambda, w)}\};\] this is called the {\it string polytope} associated to $\lambda$, $w$, and ${\bf i}$.
\end{defi}\vspace{2mm}

\begin{lem}[{see \cite[Section 1]{Lit}}]
For $(a_1, \ldots, a_r) \in \widetilde{\Delta}_{\bf i} ^{(\lambda, w)}$, there hold the inequalities \[0 \le a_r \le \langle \lambda, h_{i_r} \rangle,\ 0 \le a_{r-1} \le \langle \lambda -a_r \alpha_{i_r}, h_{i_{r-1}}\rangle, \ldots, 0 \le a_1 \le \langle\lambda -a_r \alpha_{i_r} - \cdots - a_2 \alpha_{i_2}, h_{i_1}\rangle.\] In particular, the string polytope $\widetilde{\Delta}_{\bf i} ^{(\lambda, w)}$ is bounded, and hence compact.
\end{lem}\vspace{2mm}

\begin{prop}[{see [BZ, \S\S 3.2 and Theorem 3.10]}]
Let ${\bf i} \in I^r$ be a reduced word for $w \in W$, and $\lambda$ a dominant integral weight. Then, the real closed cone $\widetilde{\mathcal{C}}_{\bf i} ^{(\lambda, w)}$ is a rational convex polyhedral cone, and the equality $\widetilde{\mathcal{S}}_{\bf i} ^{(\lambda, w)} = \widetilde{\mathcal{C}}_{\bf i} ^{(\lambda, w)} \cap (\z_{>0} \times \z^r)$ holds. In particular, the string polytope $\widetilde{\Delta}_{\bf i} ^{(\lambda, w)}$ is a rational convex polytope, and the equality $\Phi_{\bf i} (\widetilde{\mathcal{B}}_w(\lambda)) = \widetilde{\Delta}_{\bf i} ^{(\lambda, w)} \cap \z^r$ holds.
\end{prop}\vspace{2mm}

We identify the function field $\c(X(w))$ with the rational function field $\c(t_r, \ldots, t_1)$ as in \S\S 3.2, and define a valuation $\tilde{v}_{\bf i}$ on $\c(X(w))$ to be the highest term valuation on $\c(t_r, \ldots, t_1)$ with respect to the lexicographic order $t_r \succ \cdots \succ t_1$ (see Example \ref{highest term valuation}). The following is the main result of \cite{Kav}.

\vspace{2mm}\begin{prop}[{see \cite[Theorem 4.1, Corollary 4.2, and Remark 4.6]{Kav}}]\label{string polytopes}
Let ${\bf i} \in I^r$ be a reduced word for $w \in W$, and $\lambda$ a dominant integral weight. 
\begin{enumerate}
\item[{\rm (1)}] The string parameterization $\Phi_{\bf i} (b)$ is equal to $-\tilde{v}_{\bf i} (G^{\rm up} _{\lambda, w} (\pi_\lambda (b))/\tau_\lambda)$ for all $b \in \widetilde{\mathcal{B}}_w(\lambda)$.
\item[{\rm (2)}] Define the linear automorphism $\omega: \r \times \r^r \xrightarrow{\sim} \r \times \r^r$ by $\omega(k, {\bf a}) = (k, -{\bf a})$. Then, $\widetilde{\mathcal{S}}_{\bf i} ^{(\lambda, w)} = \omega(S(X(w), \mathcal{L}_\lambda, \tilde{v}_{\bf i}, \tau_\lambda))$, $\widetilde{\mathcal{C}}_{\bf i} ^{(\lambda, w)} = \omega(C(X(w), \mathcal{L}_\lambda, \tilde{v}_{\bf i}, \tau_\lambda))$, and $\widetilde{\Delta}_{\bf i} ^{(\lambda, w)} = -\Delta(X(w), \mathcal{L}_\lambda, \tilde{v}_{\bf i}, \tau_\lambda)$.
\end{enumerate}
\end{prop}\vspace{2mm}

In the rest of this subsection, we consider the case that $w = w_0$. For a reduced word ${\bf i} = (i_N, \ldots, i_1)$ for $w_0$, Lemma \ref{lemma2} and Proposition \ref{string polytopes} (1) imply that \[\Phi_{\bf i}(b) = -\tilde{v}_{\bf i}(G^{\rm up}(b))\] for all $b \in \mathcal{B}(\infty)$, where we identify $G^{\rm up}(b) \in U(\mathfrak{u}^-)_{\rm gr} ^\ast$ with $\Upsilon(G^{\rm up}(b)) \in \c[U^-]$ ($\subset \c[t_N, \ldots, t_1]$). Since $\Phi_{\bf i}(b)$, $b \in \mathcal{B}(\infty)$, are all distinct, we deduce from Proposition \ref{prop1,val} (2) that \[\Phi_{\bf i} (\mathcal{B}(\infty)) = -\tilde{v}_{\bf i} (\c[U^-] \setminus \{0\}).\] Now the equality $\Phi_{\bf i} = \Psi_{{\bf i}^{\rm op}} \circ \ast$ (by Proposition \ref{star string} (2)) implies that this set is also identical to the set \[\Psi_{{\bf i}^{\rm op}} (\mathcal{B}(\infty)) = -v_{{\bf i}^{\rm op}} (\c[U^-] \setminus \{0\})\quad({\rm by\ Corollary}\ \ref{cor1}),\] where ${\bf i}^{\rm op} := (i_1, \ldots, i_N)$, which is a reduced word for $w_0 ^{-1} = w_0$. Let $\eta_{\bf i}: \Psi_{\bf i} (\mathcal{B}(\infty)) \rightarrow \Phi_{{\bf i}^{\rm op}} (\mathcal{B}(\infty))$ ($= \Psi_{\bf i} (\mathcal{B}(\infty))$) be the transition map given by $\eta_{\bf i} (\Psi_{\bf i} (b)) = \Phi_{{\bf i}^{\rm op}} (b)$ for $b \in \mathcal{B}(\infty)$. Then, we see that Kashiwara's involution $\ast$ corresponds to the map $\eta_{\bf i}$ through the Kashiwara embedding $\Psi_{\bf i}$: \[b^\ast = \Psi_{\bf i} ^{-1} \circ \eta_{\bf i} \circ \Psi_{\bf i}(b)\] for all $b \in \mathcal{B}(\infty)$. Take an extension $\tilde{\bf i} = (\ldots, i_k, \ldots, i_{N+1}, i_N, \ldots, i_1)$ of ${\bf i}$ as in \S\S 2.3. It follows from the definition of the crystal $\z^\infty _{\tilde{\bf i}}$ that 
\begin{align*}
\tilde{e}_i ^{\rm max}{\bf a} &= \bigl(a_k - \delta_{i, i_k}\bigl(a_k - \varphi_i({\bf a}_{\ge k+1})\bigr)\bigr)_{k \ge 1}\\
&= \bigl((1 - \delta_{i, i_k}) a_k + \delta_{i, i_k} \bigl(\sigma^{(i)}({\bf a}_{\ge k+1}) + \langle {\rm wt}({\bf a}_{\ge k+1}), h_i\rangle\bigr)\bigr)_{k \ge 1}
\end{align*}
for $i \in I$ and ${\bf a} = (\ldots, a_k, \ldots, a_2, a_1) \in \z^\infty _{\tilde{\bf i}}$, where we set ${\bf a}_{\ge k} := (\ldots, a_{k+1}, a_k) \in \z^\infty _{(\ldots, i_{k+1}, i_k)}$ for $k \ge 1$, and $\tilde{e}_i ^{\rm max}{\bf a} := \tilde{e}_i ^{\varepsilon_i({\bf a})}{\bf a}$. Also, if we set ${\bf a}^\prime := (\ldots, 0, 0, a_N, \ldots, a_1) \in \z^\infty _{\tilde{\bf i}}$ for $(a_1, \ldots, a_N) \in \Psi_{\bf i} (\mathcal{B}(\infty))$, then we deduce that 
\begin{align*}
\eta_{\bf i}((a_1, \ldots, a_N)) &= \Phi_{{\bf i}^{\rm op}}(\Psi_{\tilde{\bf i}} ^{-1}({\bf a}^\prime))\\
&({\rm by\ Remark}\ \ref{definition1})\\ 
&= (\varepsilon_{i_1}(\Psi_{\tilde{\bf i}} ^{-1}({\bf a}^\prime)), \varepsilon_{i_2}(\tilde{e}_{i_1} ^{\rm max}\Psi_{\tilde{\bf i}} ^{-1}({\bf a}^\prime)), \ldots, \varepsilon_{i_N}(\tilde{e}_{i_{N-1}} ^{\rm max} \cdots \tilde{e}_{i_1} ^{\rm max}\Psi_{\tilde{\bf i}} ^{-1}({\bf a}^\prime)))\\
&({\rm by\ the\ definition\ of}\ \Phi_{{\bf i}^{\rm op}})\\
&= (\varepsilon_{i_1}({\bf a}^\prime), \varepsilon_{i_2}(\tilde{e}_{i_1} ^{\rm max}{\bf a}^\prime), \ldots, \varepsilon_{i_N}(\tilde{e}_{i_{N-1}} ^{\rm max} \cdots \tilde{e}_{i_1} ^{\rm max}{\bf a}^\prime))\\
&({\rm since}\ \Psi_{\tilde{\bf i}}\ {\rm is\ a\ strict\ embedding\ of\ crystals})\\
&= (\sigma^{(i_1)}({\bf a}^\prime), \sigma^{(i_2)}(\tilde{e}_{i_1} ^{\rm max}{\bf a}^\prime), \ldots, \sigma^{(i_N)}(\tilde{e}_{i_{N-1}} ^{\rm max} \cdots \tilde{e}_{i_1} ^{\rm max}{\bf a}^\prime))\\
&({\rm by\ the\ definition\ of\ the\ crystal\ structure\ on}\ \z^\infty _{\tilde{\bf i}}).
\end{align*}
From these, it follows that the map $\eta_{\bf i}: \Psi_{\bf i} (\mathcal{B}(\infty)) \rightarrow \Psi_{\bf i} (\mathcal{B}(\infty))$ is naturally extended to a piecewise-linear map from the string cone $\mathcal{C}_{\bf i}$ (see \S\S 2.3 for the definition) to itself, which is also denoted by $\eta_{\bf i}$; we see from the equality $\mathcal{C}_{\bf i} \cap \z^N = \Psi_{\bf i} (\mathcal{B}(\infty))$ that such an extension is unique. Since $\ast^2 =\ {\rm id}_{\mathcal{B}(\infty)}$, we see that $(\eta_{\bf i} |_{\Psi_{\bf i} (\mathcal{B}(\infty))})^2 =\ {\rm id}_{\Psi_{\bf i} (\mathcal{B}(\infty))}$, and hence that $\eta_{\bf i} ^2 =\ {\rm id}_{\mathcal{C}_{\bf i}}$. Thus, we obtain the following.

\vspace{2mm}\begin{cor}
Let ${\bf i}$ be a reduced word for $w_0$, and $\eta_{\bf i}: \mathcal{C}_{\bf i} \rightarrow \mathcal{C}_{\bf i}$ a unique piecewise-linear map such that $b^\ast = \Psi_{\bf i} ^{-1} \circ \eta_{\bf i} \circ \Psi_{\bf i}(b)$ for all $b \in \mathcal{B}(\infty)$.
\begin{enumerate}
\item[{\rm (1)}] The map $\eta_{\bf i}$ corresponds to the change of valuations from $v_{\bf i}$ to $\tilde{v}_{{\bf i}^{\rm op}}$$:$ \[\eta_{\bf i} (-v_{\bf i}(G^{\rm up} (b))) = -\tilde{v}_{{\bf i}^{\rm op}}(G^{\rm up} (b))\ {\it for\ all}\ b \in \mathcal{B}(\infty).\] 
\item[{\rm (2)}] The equality $\eta_{\bf i} ^2 = {\rm id}_{\mathcal{C}_{\bf i}}$ holds. 
\item[{\rm (3)}] The map $\eta_{\bf i}$ induces a bijective piecewise-linear map from the polyhedral realization $\Delta_{\bf i} ^{(\lambda, w_0)} = -\Delta(G/B, \mathcal{L}_\lambda, v_{\bf i}, \tau_{\lambda})$ onto the string polytope $\widetilde{\Delta}_{{\bf i}^{\rm op}} ^{(\lambda, w_0)} = -\Delta(G/B, \mathcal{L}_\lambda, \tilde{v}_{{\bf i}^{\rm op}}, \tau_{\lambda})$ for each dominant integral weight $\lambda$.
\end{enumerate}
\end{cor}\vspace{2mm}

\begin{ex}\normalfont\label{example1}
If $G = SL_3(\c)$ (of type $A_2$) and ${\bf i} = (1, 2, 1)$, a reduced word for $w_0$, then we deduce from \cite[Theorem 3.1]{NZ} that the polyhedral realization $\Psi_{\bf i} (\mathcal{B}(\infty))$ is identical to the semigroup \[\{(a_1, a_2, a_3) \in \z_{\ge 0} ^3 \mid a_2 \ge a_3\}.\] Recall that the coordinate ring $\c[U^-]$ is regarded as a $\c$-subalgebra of the polynomial ring $\c[t_3, t_2, t_1]$ by using the birational morphism
\[\c^3 \rightarrow U^-,\ (t_3, t_2, t_1) \mapsto \exp(t_3 F_1)\exp(t_2 F_2)\exp(t_1 F_1),\] where we set \[F_1 := \begin{pmatrix}
0 & 0 & 0\\
1 & 0 & 0 \\
0 & 0 & 0
\end{pmatrix},\ F_2 := \begin{pmatrix}
0 & 0 & 0\\
0 & 0 & 0 \\
0 & 1 & 0
\end{pmatrix}.\] Since we have \[\exp(t_3 F_1)\exp(t_2 F_2)\exp(t_1 F_1) =
\begin{pmatrix}
1 & 0 & 0\\
t_1 + t_3 & 1 & 0 \\
t_2t_3 & t_2 &1
\end{pmatrix},\] the coordinate ring $\c[U^-]$ is identical to the $\c$-subalgebra $\c[t_1 + t_3, t_2, t_2 t_3]$ ($\subset \c[t_3, t_2, t_1]$); hence we see from the definition of $v_{\bf i}$ that \[-v_{\bf i} (\c[U^-] \setminus \{0\}) = \{(a_1, a_2, a_3) \in \z_{\ge 0} ^3 \mid a_2 \ge a_3\},\] which is indeed identical to $\Psi_{\bf i} (\mathcal{B}(\infty))$. Now it follows from the definition of the crystal $\z^\infty _{\tilde{\bf i}}$ that 
\begin{align*}
&\varepsilon_1({\bf a}) = \max\{a_3, a_1 - a_2 + 2 a_3\},\ \tilde{e}_1 ^{\max} {\bf a} = (\min\{a_1, a_2 - a_3\}, a_2, 0),\\
&\varepsilon_2(\tilde{e}_1 ^{\max} {\bf a}) = a_2,\ \tilde{e}_2 ^{\max} \tilde{e}_1 ^{\max} {\bf a} = (\min\{a_1, a_2 - a_3\}, 0, 0),\ {\rm and}\\
&\varepsilon_1(\tilde{e}_2 ^{\max} \tilde{e}_1 ^{\max} {\bf a}) = \min\{a_1, a_2 - a_3\}
\end{align*}
for ${\bf a} = (a_1, a_2, a_3) \in \Psi_{\bf i} (\mathcal{B}(\infty))$. Therefore, we obtain \[\eta_{\bf i} ({\bf a}) = (\max\{a_3, a_1 - a_2 + 2 a_3\}, a_2, \min\{a_1, a_2 - a_3\}).\] This piecewise-linear map induces a bijection from the polyhedral realization $\Delta_{\bf i} ^{(\lambda, w_0)}$ onto the string polytope $\widetilde{\Delta}_{{\bf i}^{\rm op}} ^{(\lambda, w_0)} = \widetilde{\Delta}_{\bf i} ^{(\lambda, w_0)}$. Indeed, by using Corollary \ref{explicit description}, we deduce that \[\Delta_{\bf i} ^{(\lambda, w_0)} = \{(a_1, a_2, a_3) \in \r^3 \mid 0 \le a_1 \le \lambda_1,\ 0 \le a_3 \le \lambda_2,\ a_3 \le a_2 \le a_1 + \lambda_2\},\] where $\lambda_i := \langle \lambda, h_i\rangle$ for $i = 1, 2$. In particular, if $\lambda = \alpha_1 + \alpha_2$, then we have \[\Delta_{\bf i} ^{(\alpha_1 + \alpha_2, w_0)} = \{(a_1, a_2, a_3) \in \r^3 \mid 0 \le a_1 \le 1,\ 0 \le a_3 \le 1,\ a_3 \le a_2 \le a_1 + 1\};\] see the figure below. 
\begin{center}
\includegraphics[width=3.0cm]{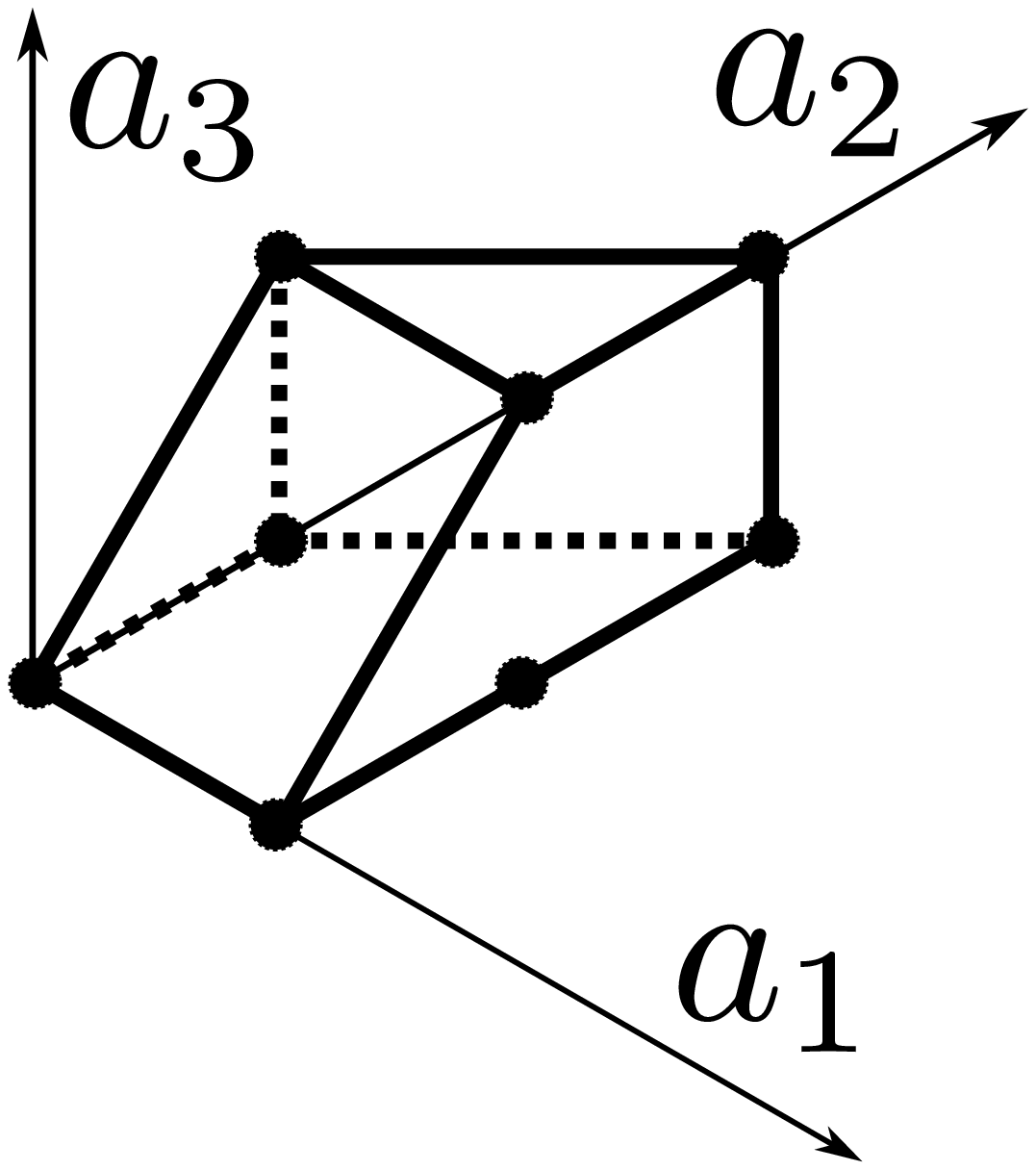}
\end{center}
Also, we deduce from \cite[Section 1]{Lit} or \cite[Theorem 3.10]{BZ} that \[\widetilde{\Delta}_{\bf i} ^{(\lambda, w_0)} = \{(a_1, a_2, a_3) \in \r^3 \mid 0 \le a_3 \le \lambda_1,\ a_3 \le a_2 \le a_3 + \lambda_2,\ 0 \le a_1 \le a_2 - 2 a_3 + \lambda_1\}.\] In particular, if $\lambda = \alpha_1 + \alpha_2$, then we have \[\widetilde{\Delta}_{\bf i} ^{(\alpha_1 + \alpha_2, w_0)} = \{(a_1, a_2, a_3) \in \r^3 \mid 0 \le a_3 \le 1,\ a_3 \le a_2 \le a_3 + 1,\ 0 \le a_1 \le a_2 - 2 a_3 +1\};\] see the figure below. 
\begin{center}
\includegraphics[width=3.0cm]{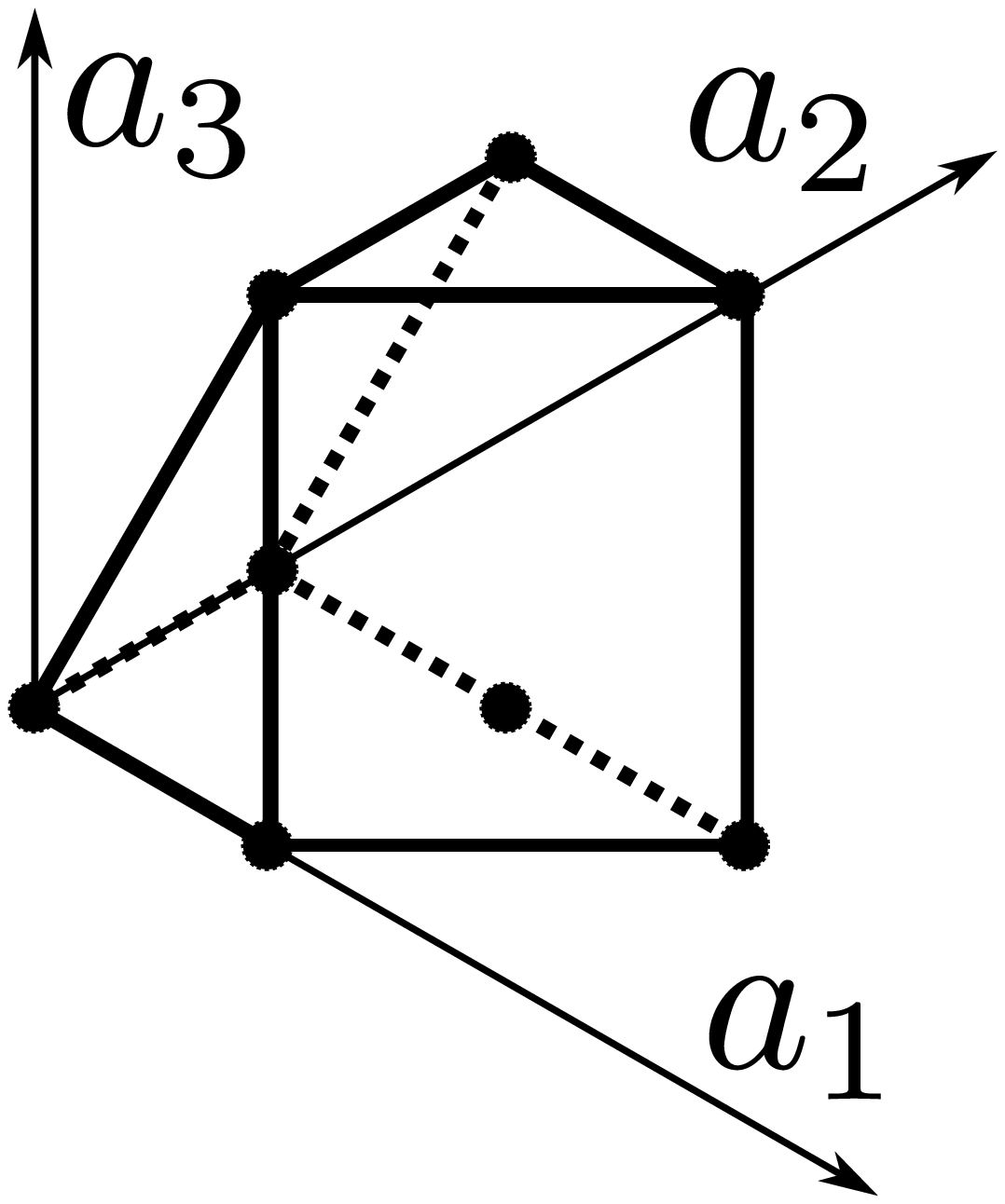}
\end{center}
\end{ex}\vspace{2mm}

\begin{ex}\normalfont
Let $G = Sp_4(\c)$ (of type $C_2$), i.e., $\langle \alpha_2, h_1 \rangle = -2$ and $\langle \alpha_1, h_2 \rangle = -1$. For a reduced word ${\bf i} = (1, 2, 1, 2)$ for $w_0$, we deduce from \cite[Theorem 3.1]{NZ} that the polyhedral realization $\Psi_{\bf i} (\mathcal{B}(\infty))$ is identical to the semigroup \[\{(a_1, a_2, a_3, a_4) \in \z_{\ge 0} ^4 \mid 2a_2 \ge a_3,\ a_3 \ge 2a_4\}.\] Now, by the same argument as in Example \ref{example1}, $\eta_{\bf i} ({\bf a})$ is given by \[(\max\{a_4, a_2 - a_3 + 2 a_4\}, \max\{a_3, a_1 - 2 a_2 + 2 a_3, a_1 + 2 a_4\}, \min\{a_2, a_3 - a_4\}, \min\{a_1, 2 a_2 - a_3, a_3 - 2 a_4\}).\] This piecewise-linear map induces a bijection from the polyhedral realization $\Delta_{\bf i} ^{(\lambda, w_0)}$ onto the string polytope $\widetilde{\Delta}_{{\bf i}^{\rm op}} ^{(\lambda, w_0)}$. Indeed, by using Corollary \ref{explicit description}, we deduce that the polyhedral realization $\Delta_{\bf i} ^{(\lambda, w_0)}$ is identical to the set of $(a_1, \ldots, a_4) \in \r^4$ satisfying the following conditions:
\begin{align*}
0 \le a_1 \le \lambda_1,\ 0 \le a_2 \le a_1 + \lambda_2,\ 0 \le a_3 \le \min\{a_2 + \lambda_2, 2 a_2\},\ 0 \le 2a_4 \le \min\{2 \lambda_2, a_3\},
\end{align*}
where $\lambda_i := \langle \lambda, h_i\rangle$ for $i = 1, 2$. In particular, if we define $\rho \in \mathfrak{t}^\ast$ by $\langle \rho, h_1\rangle = \langle \rho, h_2\rangle = 1$, then the polyhedral realization $\Delta_{\bf i} ^{(\rho, w_0)}$ is given by the inequalities:
\begin{align*}
0 \le a_1 \le 1,\ 0 \le a_2 \le a_1 +1,\ 0 \le a_3 \le \min\{a_2 +1, 2 a_2\},\ 0 \le 2a_4 \le \min\{2, a_3\}.
\end{align*}
Also, we deduce from \cite[Section 1]{Lit} or \cite[Theorem 3.10]{BZ} that the string polytope $\widetilde{\Delta}_{{\bf i}^{\rm op}} ^{(\lambda, w_0)}$ is identical to the set of $(a_1, \ldots, a_4) \in \r^4$ satisfying the following conditions: \[0 \le a_4 \le \lambda_1,\ a_4 \le a_3 \le a_4 + \lambda_2,\ a_3 \le a_2 \le 2a_3 - 2a_4 + \lambda_1,\ 0 \le a_1 \le a_2 - 2a_3 + a_4 + \lambda_2.\] In particular, if $\lambda = \rho$, then the string polytope $\widetilde{\Delta}_{{\bf i}^{\rm op}} ^{(\rho, w_0)}$ is given by the inequalities: \[0 \le a_4 \le 1,\ a_4 \le a_3 \le a_4+1,\ a_3 \le a_2 \le 2a_3 - 2a_4 +1,\ 0 \le a_1 \le a_2 - 2a_3 + a_4+1.\]
\end{ex}
\end{document}